\documentclass[a4paper,11pt,notitlepage,reqno]{amsart}

\usepackage[utf8]{inputenc}
\usepackage[spanish,english]{babel}

\usepackage[
    margin=28mm,
    marginparwidth=25mm,
    marginparsep=2mm,
    bottom=25mm
    ]{geometry}

\usepackage{amsmath,mathrsfs}
\usepackage{amsfonts}
\usepackage{amssymb}
\usepackage{amsthm}
\usepackage{xcolor}

\usepackage{tcolorbox}

\usepackage[pdfencoding=unicode, psdextra]{hyperref}

\usepackage{cleveref}

\usepackage[shortlabels]{enumitem}

\usepackage{mathtools}

\usepackage{csquotes}

\usepackage[backend=biber, style=numeric, maxbibnames=10]{biblatex} %Imports biblatex package
\DeclareNameAlias{author}{family-given} % This makes the last name first
\DeclareFieldFormat[misc]{title}{``#1''} %Add quotation marks in title in misc
\newtheorem{theorem}{Theorem}[section]

\newtheorem{proposition}[theorem]{Proposition}
\newtheorem{lemma}[theorem]{Lemma}
\newtheorem{remark}[theorem]{Remark}

\newcommand\abs[1]{\left|#1\right|}
\newcommand\norm[1]{\left\Vert#1\right\Vert}

\newcommand{\NN}{\mathbb N}
\newcommand{\ZZ}{\mathbb Z}
\newcommand{\RR}{\mathbb R}
\newcommand{\R}{\mathbb R}
\newcommand{\CC}{\mathbb C}
\newcommand{\red}{\textcolor{red}}
\newcommand{\blue}{\textcolor{blue}}
\newcommand{\edproof}{ $\hfill {\Box}$}
\newcommand{\Marta}[1]{\marginpar{\blue{\scriptsize\textbf{Marta:}#1}}}
\numberwithin{equation}{section}

\begin{document}

\title[Discrete Besov spaces, discrete Laplacian and non-local operators]{Discrete Besov spaces via semigroups associated to the discrete Laplacian and regularity of non-local operators}

\author[L. Abadias]{Luciano Abadias}
\address{\newline
    Luciano Abadias \newline
    Departamento de Matemáticas,
    Instituto Universitario de Matemáticas y Aplicaciones,
    Universidad de Zaragoza,
    50009 Zaragoza, Spain.}
\email{labadias@unizar.es}

\author[M. De León-Contreras]{Marta De León-Contreras}
\address{\newline
     Marta De León-Contreras\newline
     Departamento de Análisis Matemático,
Universidad de La Laguna, Avda. Astrofísico Sánchez, s/n 38206 La Laguna (Sta. Cruz de Tenerife), Spain.}
\email{mleoncon@ull.edu.es}

\author[A. Mahillo]{Alejandro Mahillo}
\address{\newline
    Alejandro Mahillo \newline
    Departamento de Matemáticas,
    Instituto Universitario de Matemáticas y Aplicaciones,
    Universidad de Zaragoza,
    50009 Zaragoza, Spain.}
\email{almahill@unizar.es}

\thanks{The first and third author have been partially supported by Project PID2022-137294NB-I00, DGI-FEDER, of the MCYTS and Project E48-23R, D.G. Aragón, Universidad de Zaragoza, Spain. The second author has been  partially supported by grant  PID2019-106093GB-I00 from the Spanish Government and also  by the Spanish MINECO
through  Juan de la Cierva fellowship FJC2020-044159-I}



\begin{abstract}
  In this paper we prove characterizations of the discrete Besov  spaces in terms of the heat and Poisson semigroups  associated with the discrete Laplacian that will allow us to prove regularity results for the fractional powers of the discrete Laplacian and the discrete Bessel potentials. Moreover, we provide new estimates for the derivatives of the discrete heat kernel and semigroup which are of independent interest.
\begin{comment}
  In this paper we characterize  the discrete Besov spaces in terms of the heat and Poisson semigroups  associated to the discrete Laplacian. This characterization allow us to get regularity results for the fractional powers of the discrete Laplacian in a more direct way than using the standard difference definition of the discrete Besov  spaces. Moreover, we provide new estimates for the derivatives of the discrete heat kernel and semigroup which are of independent interest.
  \end{comment}
\end{abstract}

\date{}

\maketitle
\section{Introduction}

Besov spaces are spaces of functions with certain smoothness degree which generalize  H\"older spaces and play an important role in PDEs, mathematical physics and functional analysis.  These spaces on $\RR^n$ and on domains  were introduced between 1959 and 1979 and can be viewed as real interpolation spaces in the scale of Triebel-Lizorkin spaces, see \cite{Triebel1}. Providing suitable characterizations of these spaces in the continuous setting has been a central topic of study  for many authors in the last 60 years, and there is an extensive literature in the topic, see for instance the impressive series of books \cite{Triebel1, Triebel2,Triebel3}, which collect most of the theory until 2006, and some more recent works like \cite{Bruno,BPV,BDXY,Bui,LPT,Moura,YY,YY2} and the references therein. In particular, obtaining characterizations of functional spaces in terms of the heat and/or Poisson semigroups is very convenient in order to get regularity results for non-local operators, see for instance \cite{DeLeonContrerasTorrea2021,DeLeonContrerasTorrea2020,GU,MSTZ, ST3, Taibleson2, Taibleson1964,Taibleson3}.

 Much less is known in the discrete case, and some of the main difficulties in this setting rely on the fact that the discrete heat kernel does not satisfy the usual hypothesis needed to get these type of characterizations, such as Gaussian estimates and the H\"older continuity. Despite this fact, recently, a characterization of the discrete H\"older spaces was given in terms of semigroups associated to the discrete Laplacian, see \cite{AbadiasDeLeonContreras2022}, and this characterization allows the authors to get regularity results for the fractional powers of the discrete Laplacian in a more systematic way than in \cite{CiaurriRoncalStingaTorreaVarona2018} and for a wider range of powers. {Also, in \cite{Bui2}, the author introduced the discrete Besov space $
\dot{B}_{p,1}^0(\ZZ^d)$ via the heat semigroup associated with the discrete Laplacian and proved maximal regularity estimates for the discrete parabolic heat equation.}

In this paper we prove a complete characterization of the discrete Besov spaces, that we will introduce in the following lines, in terms of the heat and Poisson semigroups associated with the discrete Laplacian. Thanks to this characterization we  obtain the regularity results in these spaces for fractional operators related to the discrete Laplacian.

For $f: \ZZ \to \RR$, consider the discrete derivatives from the right and from the left,
\[
\delta_{right} f(n) =  f(n) - f(n+1), \qquad \delta_{left} f(n) =  f(n) - f(n-1), \quad n \in \ZZ.
\]
Observe that $\delta_{right}\delta_{left} f = \delta_{left} \delta_{right} f$, so every combination of these operators
is not affected by the order when they are applied. We will use the notation $\delta_{right}^l$ and $\delta_{left}^l$ to denote the
$l$-fold composition of the operator, $l\in\NN$,  
being $\delta_{right}^0 f = f$ and $\delta_{left}^0 f = f$. Moreover, since
the $\ell^p(\ZZ)$-norms are invariant under translations, we have that for $f: \ZZ \to \RR$,
$\norm{\delta_{right}f}_p = \norm{\delta_{left} f}_p.$ Therefore,
we shall state and prove  our results  for the $\delta_{right}$ operator. %(they will remain true for the $\delta_{left}$ operator
%and any other linear combination of these operators).

 Let $\alpha > 0$ be a non-natural number, $l = [\alpha]$ the integer part of $\alpha,$ and $1 \leq p,q \leq \infty$. We define the  {\it discrete Besov spaces}, also called  discrete ``generalized'' Hölder spaces,   as 
$$C^{\alpha,p,q}(\ZZ):=\biggl\{  f:\ZZ \to \RR\,:\, \sum_{j\neq 0} \biggl\| \frac{\delta^l_{right}f(\cdot+j)-\delta^l_{right}f(\cdot) }{|j|^{\alpha-l}} \biggr\|_p^q\frac{1}{|j|}<\infty  \biggr\},\quad 1\leq q<\infty,$$ and $$C^{\alpha,p,\infty}(\ZZ):=\biggl\{  f:\ZZ \to \RR\,:\, \sup_{j\neq 0} \biggl\| \frac{\delta^l_{right}f(\cdot+j)-\delta^l_{right}f(\cdot) }{|j|^{\alpha-l}} \biggr\|_p<\infty  \biggr\}.$$

Observe that $\ell^p(\ZZ)$ functions are trivially in $C^{\alpha,p,q}(\ZZ)$. Furthermore, in Lemma \ref{lemma:C_space_bound_growth} it will be proved that  the functions belonging to these spaces satisfy  $\frac{f}{1+\abs{\cdot}^\alpha} \in \ell^q(\ZZ,\mu)$, being $\ell^q(\ZZ,\mu)$ the weighted {$\ell^q$}-space with weight  $\mu=\sum_{n\in\ZZ}\frac{1}{1+|n|} \delta_{n}$ and $\delta_n$ the Dirac measure in $n.$ 
 This size condition will be the starting point to define the following spaces of functions.
 
 For $\alpha \in \NN$ and $1 \leq p,q \leq \infty,$ we introduce the \emph{``generalized'' Zygmund classes} by
\begin{multline*}
Z^{\alpha,p,q}(\ZZ) :=  \left\{ f:\ZZ \to \RR \, : \, \frac{f}{1+\abs{\cdot}^\alpha} \in \ell^q(\ZZ,\mu) 
\text{ and }\right.\\                      
\left. \sum_{j\neq 0} \biggl\| \frac{ \delta^{\alpha-1}_{right}f(\cdot-j)- 2\delta^{\alpha-1}_{right}f(\cdot) + \delta^{\alpha-1}_{right}f(\cdot+j) }{|j|} \biggr\|_p^q\frac{1}{|j|}<\infty  \right\}, \:\; 1\leq q<\infty, 
\end{multline*}
and 
\begin{multline*}
Z^{\alpha,p,\infty}(\ZZ) :=  \left\{ f:\ZZ \to \RR \, : \, \frac{f}{1+\abs{\cdot}^\alpha} \in \ell^{\infty}(\ZZ) 
\text{ and }\right.\\                      
\left. \sup_{j\neq 0} \biggl\| \frac{ \delta^{\alpha-1}_{right}f(\cdot-j)- 2\delta^{\alpha-1}_{right}f(\cdot) + \delta^{\alpha-1}_{right}f(\cdot+j) }{|j|} \biggr\|_p <\infty  \right\}.
\end{multline*}
Observe that when $p=q=\infty$, $\ell^\infty(\ZZ,\mu)=\ell^\infty(\ZZ)$ so $C^{\alpha,\infty,\infty}(\ZZ)$ are the discrete H\"older spaces and  $Z^{\alpha,\infty,\infty}(\ZZ)$  are the discrete Zygmund spaces treated in \cite{AbadiasDeLeonContreras2022}.

For $1 \leq p\leq \infty$, we consider the mixed-norm spaces 
$$
L^{q}(((0,\infty), dt/t); \ell^p(\ZZ)) = \left\{ f: (0,\infty) \to \ell^p(\ZZ) \ :\  \norm{f}_{p,q} < \infty \right\},
$$
where
$$
\norm{f}_{p,q} = \left( \int_0^\infty \norm{f(t)}_p^q  \, \frac{dt}{t} \right)^{1/q}, \quad 1 \leq q < \infty,
$$
and
$$
\norm{f}_{p,\infty} = \inf\left\{\alpha>0\,:\, \int_{\{t>0\,:\, ||f(t)||_p>\alpha\}  }\,\frac{dt}{t}=0\right\}.
$$
It can be shown that the spaces   $L^{q}(((0,\infty), dt/t); \ell^p(\ZZ))$, $1\le p,q\le \infty,$ are Banach spaces under this norm, see \cite{BenedekPanzone1961}, and  the norms $\norm{f}_{p,\infty}$ and $\sup_{t > 0 } \norm{f(t)}_p$ coincide.

Finally, we shall introduce Besov spaces in terms of semigroups associated with the discrete Laplacian, $\Delta_d$, defined for each $f:\ZZ\to\R$ as
$$(\Delta_d  f)(n):=f(n+1)-2f(n)+f(n-1),\quad n\in\ZZ.$$ It is well-known that the heat semigroup associated with $\Delta_d$ is the solution to the discrete heat problem,
\begin{equation*}
\left\{\begin{array}{ll}
\partial_t u(t,n)-\Delta_d u(t,n)=0,&n\in\ZZ,\,t\geq 0,\\
u(0,n)=f(n),&n\in\ZZ,
\end{array} \right.
\end{equation*}
and it is given by the convolution $u(t,n)=e^{t\Delta_d}f(n):=\sum_{j\in\ZZ}G(t,n-j)f(j)=\sum_{j\in\ZZ}G(t,j)f(n-j),$ where $$G(t,n)=e^{-2t}I_{n}(2t),\quad n\in\ZZ, \: t>0,$$ being $I_n$ the modified Bessel function of the first kind and order $n\in\ZZ,$  see  Section \ref{Gaussiankernel} for more details.  
Moreover, by subordination (see \cite[Chapter IX, Section 11]{Yo}) we can define the Poisson semigroup associated with  $\Delta_d$ as \begin{align*}
    e^{-y\sqrt{-\Delta_d}}f(n)&:=\frac{y}{2\sqrt{\pi}}\int_0^\infty \frac{e^{-\frac{y^2}{4t}}}{t^{3/2}} e^{t\Delta_d}f(n)dt,\quad y>0, \: \; n\in\ZZ,%=\sum_{j\in\ZZ}\left(\frac{y}{2\sqrt{\pi}}\int_0^\infty \frac{e^{-\frac{y^2}{4t}}}{t^{3/2}} G(t,j)dt\right)f(n-j)\\  &=:\sum_{j\in\ZZ}P(y,j)f(n-j), \quad y>0, \: \; n\in\ZZ,
\end{align*}
which is the solution to the discrete Poisson problem 
\begin{equation*}\label{eqPoisson}
\left\{\begin{array}{ll}
\partial_y^2 v(y,n)-\Delta_d v(y,n)=0,&n\in\ZZ,\,y\geq 0,\\
v(0,n)=f(n),&n\in\ZZ.
\end{array} \right.
\end{equation*}

For $\alpha>0$ and $1 \leq p,q \leq \infty,$ we define the discrete heat Besov spaces and discrete Poisson Besov spaces  as
\begin{multline*}
\Lambda_H^{\alpha,p,q}:= \left\{   
f:\ZZ \to \RR \, : \, 
\frac{f}{1+\abs{\cdot}^\alpha} \in \ell^q(\ZZ,\mu) 
\text{ and } 
\norm{t^{k-\frac{\alpha}{2}} \partial_t^k e^{t\Delta_d}f}_{p,q}  < \infty, \text{ with } 
\right. 
\\
\left.
k=\left[ \frac{\alpha}{2} \right]+1, \:t>0
\right\}
\end{multline*} and
\begin{multline*}
\Lambda_P^{\alpha,p,q}:= \left\{   
f:\ZZ \to \RR \, : \, 
\sum_{n \in \ZZ} \frac{\abs{f(n)}}{1+\abs{n}^2} < \infty 
\text{ and } 
\norm{y^{l-\alpha} \partial_y^l e^{-y \sqrt{-\Delta_d}}f}_{p,q} < \infty, \text{ with } 
\right. 
\\
\left.
l=[\alpha]+1,\: y>0
\right\}.
\end{multline*}
%Inspired by the techniques  used in \cite{Taibleson1964}  for the continuous case and \cite{AbadiasDeLeonContreras2022} for the discrete H\"older classes, we will prove that the discrete Besov and the discrete  generalized Zygmund spaces are equivalent to the spaces $\Lambda_H^{\alpha,p,q}$ and $\Lambda_P^{\alpha,p,q}$. 
Our first main result is the following.

\begin{theorem}
	\label{theorem:equal_spaces_general_case}
	Let $1 \leq p, q \leq \infty$.
	\begin{enumerate}[(A)]
		\item Let $\alpha >0$.
		\begin{enumerate}[({A}1)]
			\item If $\alpha \notin \NN$, then $C^{\alpha,p,q}(\ZZ) = \Lambda_H^{\alpha,p,q}$.
			\item If $\alpha \in \NN$, then $Z^{\alpha,p,q}(\ZZ) = \Lambda_H^{\alpha,p,q}$.
		\end{enumerate}
		\item Let $f: \ZZ \to \RR$ such that $\sum_{j \in \ZZ} \frac{\abs{f(j)}}{1+\abs{j}^2} < \infty$.
		\begin{enumerate}[({B}1)]
			\item For every $\alpha>0$, $\alpha \notin \NN$,
			\[ f \in C^{\alpha,p,q}(\ZZ) \iff f \in \Lambda_H^{\alpha,p,q} \iff f \in \Lambda_P^{\alpha,p,q}. \]
			\item For every $\alpha \in \NN$,
			\[ f \in Z^{\alpha,p,q}(\ZZ) \iff f \in \Lambda_H^{\alpha,p,q} \iff f \in \Lambda_P^{\alpha,p,q}. \] 
		\end{enumerate}
	\end{enumerate}
\end{theorem}
Observe that when $p=q=\infty$, we recover the results obtained in \cite{AbadiasDeLeonContreras2022}.
In order to prove \Cref{theorem:equal_spaces_general_case}, we will need to prove some refined estimates about the derivatives of the discrete heat kernel (see Lemma \ref{lemma:inequality_heat_kernel}) {and of the discrete heat semigroup (see Lemmata \ref{lemma:decay_heat_semigroup} and \ref{lemma:inequality_norms_n_square_semigroup}), and  some mixed-norm estimates for the derivatives of the heat and Poisson semigroups (see Lemmata \ref{lemma:inequality_add_derivative}, \ref{lemma:inequality_remove_derivative}  and  \ref{lemma:inequality_interchange_derivative})}.  We believe that these results are also of independent interest.

Moreover, these characterizations of the discrete Besov spaces through the semigroup language will allow us to get the regularity results for fractional operators associated with the discrete Laplacian, such as the Bessel potentials and the powers $(-\Delta_d)^{\pm \beta}$, in a direct and systematic way. For the appropriate definition of these operators, see Section \ref{Applications}.
\begin{theorem}
	\label{theorem:BesselPotentials}
	Let $\alpha, \beta>0$, $1 \leq p,q \leq \infty$ and $f: \ZZ \to\RR$ such that $f \in \Lambda_H^{\alpha,p,q}$,
	then $(I - \Delta_d)^{-\beta / 2} f \in \Lambda_H^{\alpha+\beta,p,q}$.
\end{theorem}
In order to define the `fractional' powers of order $\beta$ of the discrete Laplacian, we need to consider the functions in the following spaces:
\[
\ell_{\pm \beta}:=\left\{u: \ZZ \to \RR: \sum_{m \in \ZZ} \frac{|u(m)|}{(1+|m|)^{1 \pm 2 \beta}}<\infty\right\}.
\]
 The choice of these spaces is justified since the fractional powers satisfy the following pointwise formula
\begin{equation*}
%\label{eq:fractionalLaplacianConvolution}
\left(-\Delta_d\right)^{\pm \beta} f(n)=\sum_{m \in \ZZ} K_{\pm \beta}(n-m) f(m), \quad n \in \ZZ,
\end{equation*}
where $K_\beta(m) \sim \frac{1}{|m|^{1+2 \beta}}$ whenever $\beta>0$ (with $K_\beta$ being of compact support if $\beta\in\NN$) and $K_{-\beta}(m) \sim \frac{1}{|m|^{1-2 \beta}}$,
for $0<$ $\beta<1 / 2$, see \cite{AbadiasDeLeonContreras2022, CiaurriRoncalStingaTorreaVarona2018}. 

Now we state our regularity results in the discrete Besov spaces for the discrete fractional Laplacian. The negative powers of the Laplacian are only
well-defined for $0<\beta<1 / 2$, since the integral that defines it is not absolutely convergent for $\beta \geq 1 / 2$. See Section \ref{Applications} for the details.

\begin{theorem}[A priori estimates]
	\label{theorem:SchauderEstimates}
	Let $\alpha>0$, $0<\beta<1 / 2$, $1 \leq p,q \leq \infty$ and $f: \ZZ \to \RR$ such that $f \in \Lambda_H^{\alpha,p,q}\cap\ell_{-\beta}$, 
	then $\left(-\Delta_d\right)^{-\beta} f \in \Lambda_H^{\alpha+2 \beta,p,q}$.
\end{theorem}

\begin{theorem}
	\label{theorem:HolderEstimates}
	Let $\alpha, \beta>0$, such that $0< 2 \beta<\alpha$, $1 \leq p,q \leq \infty$ and $f: \ZZ \to \RR$.
	\begin{enumerate}
		\item If $f \in \Lambda_H^{\alpha,p,q} \cap \ell_\beta$, then $\left(-\Delta_d\right)^\beta f \in \Lambda_H^{\alpha-2 \beta,p,q}$.
		\item If $\beta \in \NN$ and $f \in \Lambda_H^{\alpha,p,q}$, then 
		$\underbrace{(-\Delta_d) \circ (-\Delta_d) \circ \cdots \circ (-\Delta_d)}_{\beta \text{ times }}f \in \Lambda_H^{\alpha-2 \beta,p,q}$.
	\end{enumerate}
\end{theorem}

The paper is organized as follows. In Section \ref{Gaussiankernel}, we establish all the results concerning pointwise and norm estimates of the discrete heat and Poisson kernels and semigroups. Section \ref{Characterization} is devoted to prove \Cref{theorem:equal_spaces_general_case} and all the properties related to these spaces. In Section \ref{Applications} we prove the regularity results for the fractional powers of the operators. Finally, in Section \ref{Appendix} we include the Hardy inequalities in their discrete and continuous versions, that will play a crucial role in our proofs.

Throughout this article, $C$ and $c$  always denote positive constants that can change in each occurrence.

\section{Discrete heat and Poisson semigroups}\label{Gaussiankernel}
\subsection{Some known results}

In this subsection we collect some known properties about gamma and Bessel functions that we will use along the paper.

For every ${\gamma > 0},$ and $\eta>0,$ it holds that
\begin{equation}
    \label{eq:inequality_beta_type}
    (1-r)^\eta r^\gamma \leq\left(\frac{\gamma}{\gamma+\eta}\right)^\gamma, \quad  0<r<1.
\end{equation} 
This inequality was a key point in the proof of many results in \cite{AbadiasGonzalez-CamusMianaPozo2021} and
\cite{CiaurriGillespieRoncalTorreaVarona2017}. We will also use the following estimates for the Euler's gamma function (see \cite[Eq. (1)]{TricomiErdelyi1951})
%For every $\alpha, z \in \CC$,

%{\[
  %  \frac{\Gamma(z+\alpha)}{\Gamma(z)}=z^\alpha\left(1+\frac{\alpha(\alpha+1)}{2 z}+O\left(|z|^{-2}\right)\right), \quad|z| \to \infty,
%\]}
%whenever $z \neq 0,-1,-2, \ldots$ and $z \neq-\alpha,-\alpha-1, \ldots$, see \cite[Eq. (1)]{TricomiErdelyi1951}. Particularly,
\begin{equation}
    \label{eq:gamma_estimates}
    \frac{\Gamma(z+\alpha)}{\Gamma(z)}=z^\alpha\left(1+O\left(\frac{1}{|z|}\right)\right), \quad \Re z > 0, \Re \alpha>0.
\end{equation}

%\subsubsection{Bessel functions}

We denote by $I_n$ the Bessel function of imaginary argument (also called modified Bessel function of first kind) and order $n \in \ZZ$,
given by
\[
    I_n(t)=\sum_{m=0}^{\infty} \frac{1}{m ! \Gamma(m+n+1)}\left(\frac{t}{2}\right)^{2 m+n}, \quad n \in \NN_{0}:=\NN \cup \{0\},\, t \in \CC,
\]
and $I_n = I_{-n}$ for $n \in \NN$. It also has the following useful integral representation,
\[
    I_n(t) =    \frac{t^n}{\sqrt{\pi} 2^n \Gamma(n+1 / 2)} 
                \int_{-1}^1 e^{-t s}\left(1-s^2\right)^{n-1 / 2} \mathrm{~d} s, \quad n \in \NN_0,\, t \geq 0.  
\]
Likewise, for $l \in \NN_0 $ and $n\in\NN_0,$ the discrete derivatives $\delta_{right}^l I_n$ have the following representation, see \cite[Proof of Lemma 2.4]{AbadiasDeLeonContreras2022},
\begin{equation}
    \label{eq:finite_differences_bessel_function}
    \begin{aligned}
            \delta_{right}^l I_n(t) & = \frac{t^n}{\sqrt{\pi} 2^n \Gamma(n+\frac{1}{2})} 
                                        \int_{-1}^1
                                            e^{-ts}(1-s^2)^{n-\frac{1}{2}} 
                                            \left( 
                                                (s+1)^l \right. \\
                                    &       \left. 
                                                + \sum_{m=1}^{l-1} 
                                                    \frac{1}{t^m} 
                                                        \sum_{p=1}^{\min\{m,l-m\}}
                                                            d_{p,m,l} s^p (s+1)^{l-m-p}
                                            \right)
                                         \, ds,
        \end{aligned}
\end{equation}
where $d_{p,m,l} \in \RR$ are constants  only depending on $p$, $m$ and $l,$ and the sum in the second line should only be interpreted when $l>1.$

{On the other hand, the generating function of the Bessel function $I_n$ is given by
$$e^{t(x+x^{-1})/2}=\sum_{n\in\ZZ}x^n I_n(t),\quad x\neq 0,\, t\in\CC.$$
From above identity, it was proved in \cite[Theorem 3.3]{AbadiasGonzalez-CamusMianaPozo2021} that, for every $k\in\NN_0$, \begin{equation}\label{pol}
\sum_{n\in\ZZ}n^{2k}I_n(t)=e^t p_k(t),\quad \sum_{n\in\ZZ}n^{2k+1}I_n(t)=0,\quad t>0,
\end{equation}
where each $p_k(t)$ is a polynomial of degree $k$ with positive coefficients, $p_0(t)=1$,
and $p_k(0)=0$ for all $k\in\NN.$
}

\subsection{Discrete heat kernel}
The solution of the heat problem on $\ZZ$ is given by the function $e^{t\Delta_d}f(n)= \sum_{j \in \ZZ} G(t,n-j) f(j)$,
where the discrete heat kernel is given by
\[
    G(t,n) = e^{-2t}I_n(2t), \quad n \in \ZZ, t>0.
\]
In the following we state a new general estimate for the heat kernel $G$ and its discrete derivatives $\delta_{right}^l G$, refining the estimates obtained in \cite{AbadiasDeLeonContreras2022}.

\begin{lemma}
    \label{lemma:inequality_heat_kernel}
    Let $l\in \NN_0.$ Assume that one of the following statements holds 
    \begin{itemize}
        \item[(i)] $ -\left[\frac{l+1}{2}\right] -\frac{1}{2}\leq \beta$ and $t \geq 1.$ \\
        
        \item[(ii)] $ -\left[\frac{l+1}{2}\right] -\frac{1}{2}\leq \beta \leq -\frac{1}{2}$ and $t \in (0,1).$
    \end{itemize}
    Then, there exists a positive constant $C_{\beta,l}$ only depending on $\beta$ and $l$, such that
    \[
        \abs{\delta_{right}^l G(t,n)} \leq C_{\beta,l} \frac{t^\beta}{1+\abs{n}^{l+2\beta+1}},\quad n\in\ZZ.    
    \] 
\end{lemma}
\begin{proof}
Let $l\in\NN_0.$ First note that by \eqref{eq:finite_differences_bessel_function}, and performing the change of variables $1+s=\frac{u}{2t},$ we get for $n\in\NN_0$ and $t>0$ that
    
      %  \begin{align*}
       %     \delta_{right}^l I_n(t) & %= \frac{t^n}{\sqrt{\pi} 2^n \Gamma(n+\frac{1}{2})} 
                                   %     \int_{-1}^1    e^{-ts}(1-s^2)^{n-\frac{1}{2}} 
                                    %        \left( (s+1)^l \right. \\
                                    %&       \left. 
                                     %           + \sum_{m=1}^{l-1} 
                                      %              \frac{1}{t^m} 
                                       %                 \sum_{p=1}^{\min\{m,l-m\}}
                                        %                    d_{p,m,l} s^p (s+1)^{l-m-p}
                                         %   \right)\, du.
        %\end{aligned}
  %  \]
    %Taking into account that $\abs{s} \leq 1$ for $s \in [-1,1]$ and doing a change of variable we have that,
   % \[
     \begin{align*}
            \abs{\delta_{right}^l G(t,n)} & \leq \frac{C_l}{\sqrt{ t} \Gamma(n+\frac{1}{2})} \int_{0}^{4t} e^{-u} u^{n-\frac{1}{2}}
                                            \left(1-\frac{u}{4t}\right)^{n-\frac{1}{2}} \left(\left( \frac{u}{4t} \right)^l \right. \\
                                          & \left. + \sum_{m=1}^{l-1} \frac{1}{t^m} 
                                            \sum_{p=1}^{\min\{m,l-m\}} \left( \frac{u}{4t}\right)^{l-m-p} \right) \, du.
        \end{align*}
    Next, we introduce the parameter $\beta$ into this equation as follows,
    \[
        \begin{aligned}
            \abs{\delta_{right}^l G(t,n)} & \leq \frac{C_{l,\beta} t^\beta}{\Gamma(n+\frac{1}{2})} \int_{0}^{4t} e^{-u} u^{n-1-\beta}
                                            \left(1-\frac{u}{4t}\right)^{n-\frac{1}{2}} \left( \frac{u}{4t} \right)^{\frac{1}{2}+\beta}
                                            \left(\left( \frac{u}{4t} \right)^l \right. \\
                                          & \left. + \sum_{m=1}^{l-1} \frac{1}{t^m} 
                                            \sum_{p=1}^{\min\{m,l-m\}} \left( \frac{u}{4t}\right)^{l-m-p} \right) \, du=:I+II, \: \:t>0,\:n\in\NN_0.
        \end{aligned}
    \]
We are interested in getting bounds that depend on $n\in\ZZ.$ For this purpose, since we are assuming  $\beta \geq -\left[\frac{l+1}{2}\right] -\frac{1}{2}$,  in particular we have that $\beta+2l+1 >0,$ so we divide $\ZZ$ into the following three disjoint subsets
$$\ZZ=\{n\in\ZZ\,:\, n>\beta+2l+1\}\cup \{n\in\ZZ\,:\, n<-(\beta+2l+1)\}\cup \{n\in\ZZ\,:\, |n|\leq \beta+2l+1\}=:A\cup B\cup D.$$ 

First we consider that $n\in A.$ Note also that, from the condition $\beta \geq -\left[\frac{l+1}{2}\right] -\frac{1}{2}$, we have in particular that $\beta+1/2+l\geq0.$ Thus, by using  \eqref{eq:inequality_beta_type} and 
\eqref{eq:gamma_estimates} we have that, for $\beta+1/2+l>0$,
    \[
        \begin{aligned}
           |I |                &   \leq \frac{C_{l,\beta} t^\beta}{\Gamma(n+\frac{1}{2})} 
                                                                         \left(\frac{\frac{1}{2}+\beta+l}{\beta+l+n}\right)^{\frac{1}{2}+\beta+l}
                                                                         \int_0^{4t} e^{-u}u^{n-1-\beta} \, du \\
                                                                &   \leq \frac{C_{l,\beta} t^\beta}{\Gamma(n+\frac{1}{2})} 
                                                                         \frac{\Gamma(n-\beta)}{1+n^{\frac{1}{2}+\beta+l}}
                                                                    \leq \frac{C_{l,\beta} t^\beta}{1+n^{1+2\beta+l}}.      
        \end{aligned}                            
    \]
    When $\beta+1/2+l=0$ the same inequality is obtained directly.

In order to get the desired bound in $A$ for $II$, we need to make some observations. Let $l\ge 2$.  Note that 
$$
-\left[\frac{l+1}{2}\right]=\left\{\begin{array}{cc}
  -l/2,   &  \text{ if  $l$ is even}\\
      -l/2 -1/2,  &  \text{ if  $l$ is odd},
\end{array}\right.
$$
and,  for $1\le m\le l-1$, it holds that
\begin{equation}\label{min}
\min\{m,l-m\}\le  \left\{
    \begin{array}{cc}
  l/2, & \text{if $l$ is even}\\
    (l-1)/2, & \text{if $l$ is odd}.\end{array}\right.
\end{equation}
Therefore, the condition $\beta \geq -\left[\frac{l+1}{2}\right] -\frac{1}{2}$,  implies  that $\beta+1/2+l-p\geq 0,$ for $1\le p\le \min\{m,l-m\}.$ In addition, if $n\in A$ then  $n-\beta-m>0,$ with $1\le m\le l-1$. Therefore,
%In the case that $\beta > -\left[\frac{l+1}{2}\right] -\frac{1}{2}$ one always has $\beta+1/2+l-p> 0,$ and therefore 
    \[
        \begin{aligned}
           II
                                                                &   =   \frac{C_{l,\beta} t^\beta}{\Gamma(n+\frac{1}{2})}
                                                                        \sum_{m=1}^{l-1} \sum_{p=1}^{\min\{m,l-m\}}
                                                                        \int_0^{4t} e^{-u} u^{n-1-\beta-m}
                                                                            \left(1-\frac{u}{4t}\right)^{n-\frac{1}{2}}
                                                                            \left(\frac{u}{4t}\right)^{\beta+\frac{1}{2}+l-p} \, du \\
                                                                &   \leq C_{l,\beta} t^\beta \sum_{m=1}^{l-1} \sum_{p=1}^{\min\{m,l-m\}}
                                                                        \frac{\Gamma(n-\beta-m)}{\Gamma(n+\frac{1}{2})(1+n^{\beta+\frac{1}{2}+l-p})},
                                                                         \end{aligned}  \]
    where we have used \eqref{eq:inequality_beta_type}  whenever $\beta+1/2+l-p> 0$, and a direct computation when $\beta+1/2+l-p= 0$. Now by taking into account \eqref{min},  we have that every $p$  such that $1\le p\le \min\{m,l-m\}$  satisfies  $\beta+1/2+l-p=0$ if, and only if, $\beta+\left[\frac{l+1}{2}\right] +\frac{1}{2}=0$, and this holds whenever $p=l/2$ and $m=l/2$ (if $l$ is even) or $p=(l-1)/2$ and is either $m=(l-1)/2$ or $m=(l+1)/2$ (if $l$ is odd). Therefore, by using \eqref{eq:gamma_estimates} we get that 
                                                                        \begin{align*}
                                                            II    &   \leq C_{l,\beta} t^\beta \sum_{m=1}^{l-1} \sum_{p=1}^{\min\{m,l-m\}}
                                                                        \frac{1}{1+n^{l+2\beta+1+(m-p)}}
                                                                    \leq  \frac{C_{l,\beta} t^\beta}{1+n^{l+2\beta+1}}.  
        \end{align*}

Secondly, we consider $n\in B$, that is, $n < -\beta-2l-1$. Thus, $n<-l,$ and then we can write $\abs{\delta_{right}^l G(t,n)} = \abs{\delta_{right}^l G(t,\abs{n}-l)}$ for $n\in B.$ Furthermore, if $n\in B,$ we have $|n|-l>\beta+l+1,$ and in particular $|n|-l>\beta+m$ for all $m=1,\ldots, l-1.$ So, one can repeat the same steps as in the case of the subset $A$ but replacing $n$ by $\abs{n}-l$  and obtaining the same estimate.

Finally, we consider $n\in D,$ i.e., $\abs{n} \leq \beta + 2l + 1.$ Observe that $1\le  \frac{ C_{\beta,l} }{1+\abs{n}^{l+2\beta+1}}$. If $t \geq 1,$ we use \cite[Lemma 2.3]{AbadiasDeLeonContreras2022} to get
    \[
        \abs{\delta_{right}^l G(t,n)}   \leq \frac{C_{\beta,l}}{t^{\left[ \frac{l+1}{2}\right] + \frac{1}{2}}} 
                                        %= C_{\beta,l}   \frac{1+\abs{n}^{l+2\beta+1}}{1+\abs{n}^{l+2\beta+1}} 
%  \frac{t^\beta}{t^{\left[ \frac{l+1}{2}\right] + \frac{1}{2}+\beta}} 
                                        \leq C_{l,\beta} \frac{t^\beta}{1+\abs{n}^{l+2\beta+1}}.
    \]
   % where we have used the fact that in this case $n$ is bounded and $\beta \geq -\left[\frac{l+1}{2}\right] -\frac{1}{2}$.
    If $t \in (0,1)$ and $\beta \leq -\frac{1}{2}$, we use that $\abs{\delta_{right}^l G(t,n)} \leq C t^{-\frac{1}{2}},$ which implies
    the desired bound.

\end{proof}

\subsection{Heat and Poisson semigroups}

Here we present some technical lemmata  to prove our main results. Moreover, the following lemma will suggest 
the appropriate size condition that will be imposed to the functions so that  the discrete heat semigroup  is well-defined and satisfies the decay estimates   necessary to work in  $ \Lambda_H^{\alpha,p,q}$ spaces.

\begin{lemma}
    \label{lemma:C_space_bound_growth}
    Let $ \alpha>0$, $\alpha \notin \NN$, $1 \leq p, q \leq \infty$ and $f \in C^{\alpha, p, q}(\ZZ)$.
    Then, $\frac{f}{1+\abs{\cdot}^\alpha} \in \ell^q(\ZZ,\mu)$.
\end{lemma}
\begin{proof}
    We begin with the case $q = \infty$. Due to the $\ell^p(\ZZ)$ embedding, if $f \in C^{\alpha,p,\infty}(\ZZ)$
    then $f \in C^{\alpha,\infty,\infty}(\ZZ)$. Therefore, we can apply \cite[Lemma 3.1]{AbadiasDeLeonContreras2022}
    to obtain that $\frac{f}{1+\abs{\cdot}^\alpha} \in \ell^\infty (\ZZ).$

    Secondly we prove the case $1 \leq q < \infty$, and we split the proof into several cases. 
    Assume first that $\alpha \in (0,1)$. As $f \in C^{\alpha, p, q}(\ZZ)$ we have that,
   $$\norm{\frac{f}{1+\abs{\cdot}^\alpha}}_{q,\mu} \leq \norm{\frac{f-f(0)}{1+\abs{\cdot}^\alpha}}_{q,\mu} + \norm{\frac{f(0)}{1+\abs{\cdot}^\alpha}}_{q,\mu},$$ where the second summand in the above expression is finite. For the first summand, since $|f(j)-f(0)|\leq \|f(\cdot+j)-f(\cdot)\|_p$ for $j\in\ZZ,$ we have $$\norm{\frac{f-f(0)}{1+\abs{\cdot}^\alpha}}_{q,\mu}\leq \biggl( \sum_{j\neq 0} \bigg\|\frac{f(\cdot+j)-f(\cdot)}{|j|^{\alpha}}\biggr\|_{p}^q \frac{1}{|j|}\biggr)^{1/q}<\infty.$$ 
   
Now, assume that $1 < \alpha < 2$. {We have that} 
    \[
        \norm{\frac{f}{1+\abs{\cdot}^\alpha}}_{q,\mu}\leq |f(0)|+   \left(
                                                            \sum_{n=1}^\infty
                                                                \left( \frac{\abs{f(n)}}{1+n^\alpha}\right)^q \frac {1}{n}
                                                        \right)^{\frac{1}{q}}
                                                    +   \left(
                                                            \sum_{n=1}^\infty
                                                                \left( \frac{\abs{f(-n)}}{1+n^\alpha}\right)^q \frac {1}{n}
                                                        \right)^{\frac{1}{q}}
                                                =: |f(0)|+ A + B.
    \]
    By using the fact that
    \[
        \abs{f(n)} \leq \abs{f(n)-f(n-1)} + \cdots + \abs{f(1)-f(0)} + \abs{f(0)} = \sum_{j=1}^n \abs{\delta_{right}f(j-1)} + \abs{f(0)},  
    \]
    for $n \in \NN$ and by the discrete Hardy inequality (see \Cref{Hardy_discrete_inequality}) we obtain that,
    \[
        \begin{aligned}
            A   & \leq  \left(
                            \sum_{n=1}^\infty
                                \left( \sum_{j=1}^n \abs{\delta_{right}f(j-1)} \right)^q \frac{1}{n^{ \alpha q  +1}}
                        \right)^{\frac{1}{q}}
                    +   \left(
                            \sum_{n=1}^\infty
                                \left( \frac{\abs{f(0)}}{1+n^\alpha}\right)^q \frac {1}{n}
                        \right)^{\frac{1}{q}} \\
                & \leq  C \left(
                                \sum_{j=1}^\infty
                                \left( j \abs{\delta_{right}f(j-1)} \right)^q \frac{1}{j^{ \alpha q +1}}
                            \right)^{\frac{1}{q}} 
                    +   C \\
                & \leq  C \left(
                                \sum_{j=1}^\infty
                                \left( \frac{\abs{\delta_{right}f(j-1) - \delta_{right}f(-1)}}{j^{\alpha-1}}\right)^q \frac{1}{j}
                            \right)^{\frac{1}{q}}
                    +       C\left(
                                \sum_{j=1}^\infty
                                \left( \frac{\abs{\delta_{right}f(-1)}}{j^{\alpha-1}}\right)^q \frac{1}{j}
                            \right)^{\frac{1}{q}}   
                    +   C \\
                & \leq C  \left(
                                \sum_{j\neq 0}
                                \bigg\|\frac{\delta_{right}f(\cdot+j)-\delta_{right}f(\cdot)}{|j|^{\alpha-1}}\biggr\|_{p}^q\frac{1}{|j|}
                            \right)^{\frac{1}{q}}
                    +  C < \infty.
        \end{aligned}         
    \]
   For $B$ we have to consider the following inequality,
    \[
        \abs{f(-n)} \leq \abs{f(-n)-f(n+1)} + \cdots + \abs{f(-1)-f(0)} + \abs{f(0)} = \sum_{j=1}^n \abs{\delta_{right}f(-j)} + \abs{f(0)},  
    \]
    and use the same techniques as for $A$ to finish the proof for the case $1< \alpha < 2$.

    Lastly, for $\alpha > 2$ the proof follows by writing  $f$ in terms of differences of order $[\alpha]$ and by iterating the previous arguments.

\end{proof}

\begin{remark}
    \label{remark:some_semigroup_properties_work_fine}
    Let $1 \leq q \leq  \infty$ and $f:\ZZ \to \RR$ satisfying $\frac{f}{1+\abs{\cdot}^\alpha} \in \ell^q(\ZZ,\mu)$ for some
    $\alpha>0.$ %If $q=\infty$ then $\frac{f}{1+\abs{\cdot}^{\alpha}} \in \ell^\infty(\ZZ),$ and if $1 \leq q < \infty$ by
    From the embedding of the
    $\ell^q(\ZZ)$ spaces we have that $\frac{f}{1+\abs{\cdot}^{\alpha+\frac{1}{q}}} \in \ell^\infty(\ZZ)$. Therefore, 
    \cite[Lemma 2.12]{AbadiasDeLeonContreras2022} gives that the heat semigroup $e^{t \Delta_d} f$ is well-defined for every $t>0$. Furthermore, from \cite[Lemma 2.11]{AbadiasDeLeonContreras2022} it follows that $\delta_{right}^m e^{t\Delta_d}f$ and $\partial_t^l e^{t\Delta_d}f$,  $m,l\in\NN$, are well-defined,
\begin{align*}
    \delta_{right}e^{t\Delta_d}f(n)&=\sum_{j\in\ZZ} (\delta_{right}G(t,n-j))f(j)=\sum_{j\in\ZZ} G(t,j)\delta_{right}f(n-j), \quad n\in\ZZ,
\end{align*}
 and for $t=t_1+t_2,$ where $t,t_1,t_2>0,$ 
    \begin{align*}
\partial_te^{t\Delta_d}f(n)|_{t=t_1+t_2}=\sum_{j\in\ZZ}\partial_{t_1}G(t_1,j)e^{t_2\Delta_d}f(n-j)=\sum_{j\in\ZZ}G(t_1,j)\partial_{t_2}e^{t_2\Delta_d}f(n-j), \quad n\in\ZZ.
    \end{align*}
%Al previous comments will be used along the paper.
\end{remark}

Next lemma shows an estimate for the size of the heat semigroup and some conditions under the derivatives of the heat semigroup vanish at infinity.

\begin{lemma}
    \label{lemma:decay_heat_semigroup}
    Let  $1 \leq q \leq \infty$ and  $f:\ZZ \to \RR$  satisfying $\frac{f}{1+\abs{\cdot}^\alpha} \in \ell^q(\ZZ,\mu)$,
    for certain $\alpha>0$. Then,
    \begin{enumerate}
  \item    For every $t>0$, it holds that $\norm{\frac{e^{t\Delta_d}f}{1+|\cdot|^\alpha}}_{\ell^q(\ZZ,\mu)}\le C\left(1+t^{\frac{\alpha+1/q}{2}}\right)$.
 \item For every $n\in \ZZ$ and  $m, l \in \NN_0$ such that $\frac{m}{2}+l>\frac{\alpha}{2}$, we have that
    \[
        \partial_t^l\delta_{right}^{m} e^{t \Delta_d} f(n) \to 0, \quad \text{as } t \to \infty.
    \]  \end{enumerate}
\end{lemma}
\begin{proof}

First we prove \textit{(1)}. The case  $q=\infty$ was proved in \cite[Lemma 2.12 A]{AbadiasDeLeonContreras2022}. Let $1 \leq q < \infty,$ and $t>0$. By using Minkowski's inequality  we have that
\begin{align*}\norm{\frac{e^{t\Delta_d}f}{1+|\cdot|^\alpha}}_{\ell^q(\ZZ,\mu)}&=\left( \sum_{n\in\ZZ}\left|\sum_{j\in\ZZ}G(t,j)\frac{f(n-j)}{(1+|n|^\alpha)(1+|n|)^{1/q}}\right|^q\right)^{1/q}\\
&\le \sum_{j\in\ZZ}G(t,j)\left( \sum_{n\in\ZZ}\left(\frac{|f(n-j)|}{(1+|n|^\alpha)(1+|n|)^{1/q}}\right)^q\right)^{1/q}\\
&\le C G(t,0)+\sum_{j\in\ZZ\setminus\{0\}}G(t,j)|f(-j)|\\
&+\sum_{j\in\ZZ\setminus\{0\}}G(t,j)\frac{|f(0)|}{(1+|j|^\alpha)(1+|j|)^{1/q}}\\
&+\sum_{j\in\ZZ\setminus\{0\}}G(t,j)\left( \sum_{n\in\ZZ\setminus\{0,j\}}\left(\frac{|f(n-j)|}{(1+|n|^\alpha)|n|^{1/q}}\right)^q\right)^{1/q}\\
&=:I+II+III+IV.
\end{align*}
Next, we work with each summand above separately. Note that the $I$ is bounded.

Secondly, since $\frac{f}{1+\abs{\cdot}^\alpha} \in \ell^q(\ZZ,\mu)$ implies that $\frac{f}{1+\abs{\cdot}^{\alpha+1/q}} \in \ell^\infty(\ZZ)$, by taking $m$ as the smallest integer such that $2m>\alpha+1/q,$ we have that \begin{align*}II=\sum_{j\in\ZZ\setminus\{0\}}G(t,j)|f(-j)|&\leq C \norm{\frac{f}{1+\abs{\cdot}^{\alpha+1/q}}}_\infty \sum_{j\in\ZZ\setminus\{0\}}G(t,j)(1+ |j|^{\alpha+1/q})\\
&\leq C \norm{\frac{f}{1+\abs{\cdot}^{\alpha+1/q}}}_\infty \biggl(1+\sum_{|j|\leq \sqrt{t}}G(t,j)|j|^{\alpha+1/q}\\
&+\sum_{|j|> \sqrt{t}}G(t,j)|j|^{\alpha+1/q}\min\left\{\frac{|j|}{\sqrt{t}},|j|\right\}^{2m-\alpha-1/q}\biggr)\\
&\leq C \norm{\frac{f}{1+\abs{\cdot}^{\alpha+1/q}}}_\infty \biggl(1+t^{\frac{\alpha+1/q}{2}}+Cp_m(2t)\min\left\{\frac{1}{t^{m-\frac{\alpha+1/q}{2}}},1\right\}\biggr)\\
&\leq C\norm{\frac{f}{1+\abs{\cdot}^{\alpha+1/q}}}_\infty (1+t^{\frac{\alpha+1/q}{2}}),
\end{align*}
where we have used that $\|G(t,\cdot)\|_1=1$ and that $|p_m(2t)|\leq C$ for $0<t<1,$ and $|p_m(2t)|\leq Ct^m$ for $t\geq 1,$ see \eqref{pol}.
\begin{comment}

\red{VERSION ANTERIOR: 
Secondly, by H\"older's inequality and $\frac{f}{1+\abs{\cdot}^\alpha} \in \ell^q(\ZZ,\mu)$, if $m$ is the smallest integer such that $2m>\alpha+1/q,$ we have that \begin{align*}II=\sum_{j\in\ZZ\setminus\{0\}}G(t,j)|f(-j)|&\leq C \biggl(\sum_{j\in\ZZ\setminus\{0\}}|G(t,j)|^p |j|^{p(\alpha+1)-1}\biggr)^{1/p}\\
&\leq C \biggl(\sum_{|j|\leq \sqrt{t}}|G(t,j)|^p |j|^{p(\alpha+1)-1}\biggr)^{1/p}\\
&+C\biggl(\sum_{|j|> \sqrt{t}}|G(t,j)|^p |j|^{p(\alpha+1)-1}\min\{\frac{|j|}{\sqrt{t}},|j|\}^{(2m-\alpha-1/q)p}\biggr)^{1/p}\\
&\leq Ct^{\frac{\alpha+1/q}{2}}+Cp_m(2t)\min\{\frac{1}{t^{2m-\alpha-1/q}},1\}\\
&\leq C(1+t^{\frac{\alpha+1/q}{2}}),
\end{align*}
where in the two last inequalities we have used $\ell^p(\ZZ)\subset \ell^1(\ZZ)$ and $\|G(t,\cdot)\|_1=1,$ and that $|p_m(2t)|\leq C$ for $0<t<1,$ and $|p_m(2t)|\leq Ct^m$ for $t\geq 1.$ }
\end{comment}

 For $III,$ since $|j|\ge 1,$ we get that $III\le C \|G(t,\cdot)\|_1=C$.

\begin{comment}
\red{VERSION ANTERIOR: For $III,$ recall that it was proved in \cite[Lemma 4.1 (i)]{AbadiasGonzalez-CamusMianaPozo2021} that
$$
G(t,j)\le C\left\{\begin{array}{cc}
\frac{t}{|j|^3},  & \text{ if } t\le |j|^2\\
t^{-1/2},  & \text{ if } t>|j|^2,
\end{array}\right.
$$  
so $G(t,j)\le \frac{C}{|j|},$  for $j\in\ZZ\setminus\{0\}$ and $t>0.$ Then $III$ is bounded.}
\end{comment}

Finally, note that if $j\neq 0$ and $n\neq 0,j,$ we have $$\frac{(1+|n-j|^\alpha)|n-j|^{1/q}}{(1+|n|^\alpha)|n|^{1/q}}\leq C (1+|j|^\alpha)|j|^{1/q},$$ 
so by proceeding as in  the case $II,$ we get
$$IV\leq C\sum_{j\in\ZZ\setminus\{0\}}G(t,j)(1+|j|^\alpha)|j|^{1/q}\leq C\left(1+t^{\frac{\alpha+1/q}{2}}\right).$$

Now we shall prove  \textit{(2).} The case $q=\infty$ was proved in \cite[Lemma 2.13]{AbadiasDeLeonContreras2022}.

   Let $1 \leq q < \infty$, $m, l \in \NN_0$ such that $\frac{m}{2}+l>\frac{\alpha}{2}$  and $n\in\ZZ$. Since the semigroup is the solution to the heat equation, we can write
    \[
       \abs{\partial_t^l\delta_{right}^m e^{t \Delta_d} f(n)}  = \abs{\delta_{right}^{2l+m} e^{t \Delta_d} f(n-l)}.
    \]
    Now %by using the fact that $\frac{m}{2}+l>\frac{\alpha}{2},$ 
    we choose $\varepsilon>0$ small enough such that $\alpha -2l -m + \epsilon < 0$.
    Then, by using \Cref{lemma:inequality_heat_kernel} for $t\ge 1$ and $\beta = \frac{\alpha -2l -m + \varepsilon}{2}$, %so as $t$ will go to infinity we can assume that $t \geq 1$ and use 
    % obtain the following expression,
    we obtain that
    \[
        \begin{aligned}
            \abs{\partial_t^l\delta_{right}^m e^{t \Delta_d} f(n)} 
                & \leq C_{l,m,\alpha} \sum_{j \in \ZZ} \frac{t^{\beta} \abs{f(n-l-j)}}{1+\abs{j}^{\alpha+\varepsilon+1}} \\
                & = C_{l,m,\alpha} t^\beta  \sum_{j \in \ZZ} \frac{\abs{f(n-l-j)}(1+|n-l-j|^{\alpha})}{(1+|n-l-j|^{\alpha})(1+\abs{j}^{\alpha+\varepsilon+1})}.
                % & = C_{l,m,\alpha} t^\beta  \sum_{j \in \ZZ} \frac{\abs{f(n-l-j)}(1+|n-l-j|^{\alpha})(1+|n-l-j|)^{1/q}}{(1+|n-l-j|^{\alpha})(1+|n-l-j|)^{1/q}(1+\abs{j}^{\alpha+\varepsilon+1})}.
        \end{aligned}
    \]
   Hölder's inequality gives the convergence for the  series above and since $\beta<0$ it follows that $\abs{\partial_t^l\delta_{right}^m e^{t \Delta_d} f(n)} \to 0$, as $t \to \infty$.
\end{proof}

The following lemma  provides the size condition for the sequence $e^{\cdot^2 \Delta_d} f(\cdot): \ZZ \to \RR$, given by $e^{n^2 \Delta_d} f(n) = \sum_{j \in \ZZ } G(n^2,n-j) f(j)$,
 $n \in \ZZ$, which will be a key point in the proof of one of our main results.  

\begin{lemma}
    \label{lemma:inequality_norms_n_square_semigroup}
    Let $1\leq q < \infty$, $l \in \NN_{0}$, $\alpha \geq l$ and $f:\ZZ \to \RR$ such that $\frac{f}{1+|\cdot|^{\alpha}}\in \ell^q(\ZZ,\mu)$. Then, $\frac{e^{\cdot^2 \Delta_d} \delta_{right}^l f(\cdot)}{1+\abs{\cdot}^{\alpha-l}}\in \ell^q(\ZZ,\mu).$
\end{lemma}
\begin{proof}
First, recall that $I_j(0)=\delta_0(j)$ for $j\in\ZZ,$ (where $\delta_0$ denotes the Dirac delta sequence), so $\delta_{right}^lG(0,-j)=(-1)^j\binom{l}{j},$ for $j=0,\ldots,l,$ and $\delta_{right}^l G(0,-j)=0$ in another case. Then, 
\[
    \begin{array}{l}
\displaystyle\norm{\frac{e^{\cdot^2 \Delta_d} \delta_{right}^l f(\cdot)}{1+\abs{\cdot}^{\alpha-l}}}_{q,\mu} 
                 \le \left(
                            \sum_{n\in\ZZ}
                                \left(
                                    \sum_{j\in\ZZ}
                                    \frac{\abs{\delta_{right}^lG(n^2,n-j)}\abs{f(j)}}{1+\abs{n}^{\alpha-l}}
                                \right)^q
                                \mu(n)
                        \right)^{\frac{1}{q}} \\ \\
                 \displaystyle\leq  C_{l}
                        \max\{|f(j)|\,:\, j=0,\ldots,l\}+\left( 
                            \sum_{n\neq 0}
                                \left(
                                    \sum_{j\in\ZZ}
                                    \frac{\abs{\delta_{right}^lG(n^2,n-j)}\abs{f(j)}}{1+\abs{n}^{\alpha-l}}
                                \right)^q
                                \mu(n)
                        \right)^{\frac{1}{q}}.

    \end{array}
   \]
So, we have to prove the convergence of the last summand. For this purpose, we consider next disjoint partition 
\begin{align*}
&\{(n,j)\,:\, n\in\ZZ\setminus \{0\},j\in\ZZ\}\\&\qquad=\{(n,j)\,:\, n\geq 1,j\geq 1\}\cup \{(n,j)\,:\, n\geq 1,j\leq -1\}\cup \{(n,j)\,:\, n\geq 1,j=0\}\\ 
&\qquad\cup \{(n,j)\,:\, n\leq  -1,j\geq 1\}\cup \{(n,j)\,:\, n\leq -1,j\leq -1\}\cup \{(n,j)\,:\, n\leq -1,j=0\}\\&\qquad =:A.1\cup\ldots \cup A.6.
\end{align*}

By \Cref*{lemma:inequality_heat_kernel}, we have that 
    $\abs{\delta_{right}^lG(n^2,n-j)} \leq C_{\beta,l} \frac{\abs{n}^{2\beta}}{1+\abs{n-j}^{2\beta+l+1}}$, for $n \in \ZZ\setminus\{0\},$ 
     $j \in \ZZ$ and any  $\beta \geq - [\frac{l+1}{2}]-\frac{1}{2}.$ 
     %In the following, we will apply above inequality for different specific values of $\beta.$ 

%First, taking $\beta = -\frac{l+1}{2}$ we have 

On the one hand, we have that $$\abs{\delta_{right}^lG(n^2,n)} \leq C_{\beta,l} \abs{n}^{-(l+1)},\quad (n,0)\in A.3,A.6,$$ so
$$\left( \sum_{n\neq 0}\left(\frac{\abs{\delta_{right}^lG(n^2,n)}\abs{f(0)}}{1+\abs{n}^{\alpha-l}}
                                \right)^q
                                \mu(n)
                        \right)^{\frac{1}{q}}\leq C_{\beta,l}\abs{f(0)}
            \left(
                \sum_{n=1}^\infty
                    \left(
                        \frac{n^{-l-1}}{(1+n^{\alpha-l})}
                    \right)^q
                    \frac{1}{n}
            \right)^{\frac{1}{q}}
        \leq C_{\beta,l} \abs{f(0)}.$$

%Let now split $A.1$ as follows $$A.1=\{(n,j)\in A.1\,:\, 1\leq j\leq n\}\cup \{(n,j)\in A.1\,:\, n+1\leq j\leq 2n\}\cup \{(n,j)\in A.1\,:\, j\geq 2n+1\}.$$ In the first two previous subsets of $A.1$ we take 
Now, by taking $\beta = -\frac{l+1}{2},$ we get  $$\abs{\delta_{right}^lG(n^2,n-j)} \leq C_{\beta,l} \abs{n}^{-(l+1)},\quad 1\le j\le 2n.$$ 
Thus, by using Hardy's inequality (see \Cref{Hardy_discrete_inequality}) we have that \begin{align*}
\left( \sum_{n=1}^{\infty}\left(\sum_{j =1}^{n}\frac{\abs{\delta_{right}^lG(n^2,n-j)}\abs{f(j)}}{1+\abs{n}^{\alpha-l}}
                                \right)^q
                                \mu(n)
                        \right)^{\frac{1}{q}}&\leq  C_{\beta,l}  \left(
                        \sum_{n=1}^\infty
                            \left(
                                \sum_{j=1}^n
                                \frac{|f(j)|}{n^{\alpha+1}}
                            \right)^q
                        \frac{1}{n}
                    \right)^{\frac{1}{q}}
            \\
            &\leq C_{\beta,l}  \left(
                        \sum_{n=1}^\infty
                            \left(
                                \frac{|f(n)|}{n^\alpha}
                            \right)^q
                        \frac{1}{n}
                    \right)^{\frac{1}{q}}
           \\& \leq C_{\beta,l} \norm{\frac{f(\cdot)}{1+\abs{\cdot}^\alpha}}_{q,\mu},
\end{align*}
and by  direct computations  we get that
\begin{align*}
\left( \sum_{n=1}^{\infty}\left(\sum_{j =n+1}^{2n}\frac{\abs{\delta_{right}^lG(n^2,n-j)}\abs{f(j)}}{1+\abs{n}^{\alpha-l}}
                                \right)^q
                                \mu(n)
                        \right)^{\frac{1}{q}}&\leq C_{\beta,l}  \left(
                            \sum_{n=1}^\infty
                                \left(
                                    \sum_{j=n+1}^{2n}
                                    \frac{|f(j)|}{n^{\alpha+1}}
                                \right)^q
                            \frac{1}{n}
                        \right)^{\frac{1}{q}}
                    \\
                    &\leq C_{\beta,l}  \left(
                            \sum_{n=1}^\infty
                                \left(
                                    \sum_{j=n+1}^{2n}
                                    \frac{|f(j)| n}{j^{\alpha+2}}
                                \right)^q
                            \frac{1}{n}
                        \right)^{\frac{1}{q}} \\
                &   \leq C_{\beta,l}  \left(
                            \sum_{n=1}^\infty
                                \left(
                                    \frac{|f(n)| n^2}{n^{\alpha+2}}
                                \right)^q
                            \frac{1}{n}
                        \right)^{\frac{1}{q}}\\&
                    \leq C_{\beta,l} \norm{\frac{f(\cdot)}{1+\abs{\cdot}^\alpha}}_{q,\mu}.
\end{align*}

%For the last splitting subset of $A.1,$ 
On the other hand, by using \Cref{lemma:inequality_heat_kernel} with  $\beta = \frac{\alpha-l+1}{2}$ one gets $$\abs{\delta_{right}^lG(n^2,n-j)} \leq C_{\beta,l} \frac{\abs{n}^{\alpha-l+1}}{1+(j-n)^{\alpha+2}}, \quad j\ge 2n+1.$$ 
Notice that in this  case $j-n>\frac{j}{2}$, so by Hardy's inequality (see \Cref{Hardy_discrete_inequality}) we have that
\begin{align*}
\left( \sum_{n=1}^{\infty}\left(\sum_{j =2n+1}^{\infty}\frac{\abs{\delta_{right}^lG(n^2,n-j)}\abs{f(j)}}{1+\abs{n}^{\alpha-l}}
                                \right)^q
                                \mu(n)
                        \right)^{\frac{1}{q}}&\leq C_{\beta,l}  \left(
                                \sum_{n=1}^\infty
                                    \left(
                                        \sum_{j=2n+1}^\infty
                                        \frac{|f(j)| n }{(j-n)^{\alpha+2}}
                                    \right)^q
                                \frac{1}{n}
                            \right)^{\frac{1}{q}}\\
                    &\leq C_{\beta,l}  \left(
                                \sum_{n=1}^\infty
                                    \left(
                                        \sum_{j=2n+1}^\infty
                                        \frac{|f(j)| n}{j^{\alpha+2}}
                                    \right)^q
                                \frac{1}{n}
                            \right)^{\frac{1}{q}} \\
                &   \leq C_{\beta,l} \norm{\frac{f(\cdot)}{1+\abs{\cdot}^\alpha}}_{q,\mu}.
\end{align*}

To finish the proof, we  can proceed in the same way for the subsets $A.2,A.4, A.5$, by taking $|j|,|n|$ instead of $j,n,$ and by taking into account that $|n-j|\geq ||n|-|j||.$ %Following the same calculations one proves the result.
\end{proof}

{Now include some lemmata which state mixed-norm estimates for the derivatives of the heat and Poisson semigroups that provide alternative semigroup conditions to characterize the discrete heat and Poisson Besov spaces.}

\begin{lemma}
    \label{lemma:inequality_add_derivative}
    Let $\beta > 0 $, $1\leq p,q \leq \infty$ and $f:\ZZ \to \RR$.
    \begin{itemize}
        \item[(i)] Suppose that $\frac{f}{1+\abs{\cdot}^\alpha} \in \ell^q(\ZZ,\mu)$ for some $\alpha > 0$. 
        If $m, l \in \NN_0$, such that $\frac{m}{2}+l>\frac{\alpha}{2},$ then
        \[
            \norm{t^{\beta} \partial_t^l \delta_{right}^m e^{t \Delta_d }f}_{p,q} 
            \leq 
            \frac{1}{\beta} \norm{t^{\beta+1} \partial_t^{l+1} \delta_{right}^m e^{t \Delta_d }f}_{p,q},\quad t>0.
        \]
        \item[(ii)] Suppose that $f$ satisfies $\sum_{j \in \ZZ} \frac{\abs{f(j)}}{1+\abs{j}^2} < \infty$. If $m, l\in \NN_0$,
        such that $m+l \geq 1,$ then
        \[
            \norm{y^{\beta} \partial_y^l \delta_{right}^{m} e^{-y\sqrt{-\Delta_d}}f}_{p,q} 
            \leq 
            \frac{1}{\beta} \norm{y^{\beta+1} \partial_y^{l+1} \delta_{right}^{m} e^{-y\sqrt{-\Delta_d}}f}_{p,q},\quad y>0.
        \]
    \end{itemize}
\end{lemma}
\begin{proof}
    The proof of this result runs parallel to the one of \cite[Lemmata 4  c) and $4^*$ c)]{Taibleson1964}. In our case, we also have to use 
    \Cref{lemma:decay_heat_semigroup}, so that $\partial_t^l \delta_{right}^m e^{t\Delta_d}f(n) \to 0,$ as $t \to \infty$, and \cite[Lemma 2.13]{AbadiasDeLeonContreras2022}, so that $\partial_y^l \delta_{right}^m e^{-y\sqrt{-\Delta_d}}f(n) \to 0$, as $y \to \infty.$
\end{proof}

\begin{lemma}
    \label{lemma:inequality_remove_derivative}
    Let $\beta > 0$, $1\leq p,q \leq \infty$, $m,l \in \NN_0$, and $f:\ZZ \to \RR$.
    \begin{itemize}
        \item[(i)] Suppose that $\frac{f}{1+\abs{\cdot}^\alpha} \in \ell^q(\ZZ, \mu)$ for some $\alpha > 0$.
        \begin{itemize}
        
         \item If $l\in\NN,$ there is $C>0$ such that $$\norm{t^{\beta+1} \partial_t^l \delta_{right}^{m} e^{t \Delta_d}f}_{p,q}  \leq C \norm{t^{\beta} \partial_t^{l-1} \delta_{right}^{m} e^{t \Delta_d}f}_{p,q},\quad t>0.$$

         \item If $m\in\NN,$ there is $C>0$ such that $$\norm{t^{\beta+\frac{1}{2}} \partial_t^l \delta_{right}^{m} e^{t \Delta_d}f}_{p,q}\leq 
                 C \norm{t^{\beta} \partial_t^{l} \delta_{right}^{m-1} e^{t \Delta_d}f}_{p,q},\quad t>0.$$
         
         \end{itemize}

        \item[(ii)] Suppose $f$ satisfies $\sum_{j \in \ZZ} \frac{\abs{f(j)}}{1+\abs{j}^2} < \infty.$ 

        \begin{itemize}
        
         \item If $l\in\NN,$ there is $C>0$ such that $$\norm{y^{\beta+1} \partial_y^l \delta_{right}^{m} e^{-y\sqrt{-\Delta_d}}f}_{p,q}  \leq 
                 C \norm{y^{\beta} \partial_y^{l-1} \delta_{right}^{m} e^{-y\sqrt{-\Delta_d}}f}_{p,q},\quad y>0.$$

         \item If $m\in\NN,$ there is $C>0$ such that $$\norm{y^{\beta+1} \partial_y^l \delta_{right}^{m} e^{-y\sqrt{-\Delta_d}}f}_{p,q} \leq 
                 C \norm{y^{\beta} \partial_y^{l} \delta_{right}^{m-1} e^{-y\sqrt{-\Delta_d}}f}_{p,q},\quad y>0.$$
         
         \end{itemize}

    \end{itemize}  
\end{lemma}
\begin{proof}
    Again, the proof of this result runs parallel to the one of \cite[Lemmata 4  a), b) and $4^*$  a), b)]{Taibleson1964}. In this case, we have used \cite[Lemmata 2.6, 2.9 and 2.11, and Remarks 2.7 and 2.10]{AbadiasDeLeonContreras2022}.
\end{proof}

\begin{remark}\label{remark:amount_of_derivatives_on_t}
From Lemmata \ref{lemma:inequality_add_derivative} and \ref{lemma:inequality_remove_derivative} we deduce that if $f\in \Lambda^{\alpha,p,q}_H$ for some $\alpha>0,$ $1\le p,q\le \infty$ and $k,l$ are natural numbers such that $k,l\ge[\alpha/2]+1$, then    $  \norm{t^{k-\frac{\alpha}{2}} \partial_t^k e^{t \Delta_d}f }_{p,q}<\infty$ if, and only if,
                $ \norm{t^{l-\frac{\alpha}{2}} \partial_t^l e^{t \Delta_d}f }_{p,q} < \infty$. {Analogously, if $k,l$ are natural numbers such that $k,l\ge[\alpha]+1$, then    $\norm{y^{k-\alpha} \partial_y^k e^{-y \sqrt{-\Delta_d}}f }_{p,q}<\infty$ if, and only if,
                $ \norm{y^{l-\alpha} \partial_y^l e^{-y \sqrt{-\Delta_d}}f }_{p,q} < \infty$.}\end{remark}
\begin{lemma}
    \label{lemma:inequality_interchange_derivative}
    Let $\beta > 0 $, $1\leq p,q \leq \infty$ and $f:\ZZ \to \RR$.
    \begin{itemize}
        \item[(i)] Suppose that $\frac{f}{1+\abs{\cdot}^\alpha} \in \ell^q(\ZZ, \mu)$ for some $\alpha > 0,$ and  $m,l \in \NN_0$ such that $\frac{m}{2}+l>\frac{\alpha}{2}.$

\begin{itemize}
        
         \item If $l\in\NN,$ there is $C>0$ such that $$\norm{t^{\beta+\frac{1}{2}} \partial_t^l \delta_{right}^m e^{t\Delta_d}f}_{p,q}  \leq 
                    C
                    \norm{t^{\beta} \partial_t^{l-1} \delta_{right}^{m+1} e^{t\Delta_d}f}_{p,q},\quad t>0.$$

         \item If $m\in\NN,$ there is $C>0$ such that $$\norm{t^{\beta} \partial_t^l \delta_{right}^m e^{t\Delta_d}f}_{p,q} \leq 
                    C
                    \norm{t^{\beta+\frac{1}{2}} \partial_t^{l+1}\delta_{right}^{m-1} e^{t\Delta_d}f}_{p,q},\quad t>0.$$
         
         \end{itemize}

        \item[(ii)] Suppose that $f$ satisfies $\sum_{j \in \ZZ} \frac{\abs{f(j)}}{1+\abs{j}^2} < \infty,$ and $m,l \in \NN_0.$
        \begin{itemize}
        
         \item If $l\in\NN,$ there is $C>0$ such that $$\norm{y^{\beta} \partial_y^l \delta_{right}^m e^{-y\sqrt{-\Delta_d}}f}_{p,q}  \leq 
                C
                \norm{y^{\beta} \partial_y^{l-1} \delta_{right}^{m+1} e^{-y\sqrt{-\Delta_d}}f}_{p,q},\quad y>0.$$

         \item If $m\in\NN,$ there is $C>0$ such that $$\norm{y^{\beta}\partial_y^l \delta_{right}^m e^{-y\sqrt{-\Delta_d}}f}_{p,q}  \leq 
                    C
                \norm{y^{\beta} \partial_y^{l+1}\delta_{right}^{m-1} e^{-y\sqrt{-\Delta_d}}f}_{p,q},\quad y>0.$$
         
         \end{itemize}

    \end{itemize}
\end{lemma}
\begin{proof}
    This result follows directly by using Lemmata \ref{lemma:inequality_remove_derivative} and  \ref{lemma:inequality_add_derivative} and the ideas presented in the 
    proof of \cite[Theorem 1]{Taibleson1964}.
\end{proof}

The last lemma of this section states an inequality involving both, heat and Poisson semigroups. The proof  follows similar steps than the ones appearing in \cite[Theorem 4.1]{DeLeonContrerasTorrea2020} and \cite[Theorem 5.6]{DeLeonContrerasTorrea2021}.

\begin{lemma}
    \label{lemma:inequalities_integral}
    Let $\alpha>0$, $1 \leq p \leq \infty$ and $f:\ZZ \to \RR$ such that $\frac{f}{1+\abs{\cdot}^\alpha} \in \ell^\infty(\ZZ)$ and
    $\sum_{n\in\ZZ} \frac{\abs{f(n)}}{1+n^2}<\infty$. Then, for every  $k \in \NN$,
    \[
        \norm{\partial_y^{2k} e^{-y\sqrt{-\Delta_d}}f}_p \leq  \int_0^\infty 
                                                                \frac{1}{2 \pi}
                                                                \frac{y e^{ \frac{-y^2}{4t}}}{t^{\frac{3}{2}}}
                                                                \norm{\partial_t^k e^{t\Delta_d}f}_p
                                                            \, dt.  
    \]
\end{lemma}

\section{Characterization via semigroups of discrete Besov spaces}\label{Characterization}
In the following we prove the characterization of $C^{\alpha,p,q}(\ZZ)$   by using the spaces
$\Lambda_H^{\alpha,p,q}$ and $\Lambda_P^{\alpha,p,q}$.

\subsection{Case \texorpdfstring{$0 < \alpha < 1$}{0 < α < 1}}

\begin{proposition}
    \label{prop:C_implies_Lambda_H_case_0_1}
    Let $0 < \alpha < 1$ and $1 \leq p,q \leq \infty.$ If $f \in C^{\alpha,p,q}(\ZZ)$ then $f \in \Lambda_H^{\alpha,p,q}$.
\end{proposition}
\begin{proof}
Let  $f \in C^{\alpha,p,q}(\ZZ)$.
    From \Cref{lemma:C_space_bound_growth} we have that  
    $\frac{f}{1+\abs{\cdot}^\alpha} \in \ell^q(\ZZ, \mu)$.
    Now we have to prove that $\norm{t^{1-\frac{\alpha}{2}} \partial_t e^{t\Delta_d}f}_{p,q}<\infty.$ For that aim,  we shall consider the $p$-norm of the derivative of the semigroup.  Note that $\partial_t G(t,j) = \Delta_d G(t,j) = \delta_{right}^2G(t,j-1)$, for $j\in\ZZ$, and $ \delta_{right}^2G(t,j-1) =  \delta_{right}^2 G(t,\abs{j}-1)$,
    for $j \leq -1$. Suppose that   $1 \leq p \leq \infty$ and
    $1 \leq q < \infty$.  If $1 \leq p < \infty$,  by Minkowski's integral inequality one gets
    
        \begin{eqnarray}\label{4.1}
           \nonumber \norm{\partial_t e^{t\Delta_d}f}_p  & =& \left(\sum_{n\in\ZZ} \abs{ \sum_{j\in\ZZ} \partial_t G(t,j) f(n\!-\!j)}^p \right)^\frac{1}{p}
                                                \leq\sum_{j\in\ZZ}\left(\sum_{n\in\ZZ}\abs{\partial_t G(t,j)(f(n-j)-f(n))}^p\right)^\frac{1}{p}\\
                                                 & =& \sum_{j\in\ZZ} \abs{\partial_t G(t,j)} \norm{f(\cdot)-f(\cdot+j)}_p
                                                  \\
                                                  \nonumber&=& 2 \sum_{j \geq 1} \abs{\delta_{right}^2 G(t,j-1)} \norm{f(\cdot)-f(\cdot+j)}_p \\
                 \nonumber                               & =& 2 \sum_{j \geq 0} \abs{\delta_{right}^2 G(t,j)} \Omega_p(j+1),
        \end{eqnarray}
    
    with $\Omega_p(j) = \norm{f(\cdot) - f(\cdot+j)}_p$. If $p=\infty$, we also have
    \[
        \begin{aligned}
            \norm{\partial_t e^{t\Delta_d}f}_\infty & = \sup_{n \in \ZZ} \abs{\sum_{j\in\ZZ} \partial_t G(t,j) f(n-j)}
                                                    \leq \sum_{j \in \ZZ} \abs{\partial_t G(t,j)} \norm{f(\cdot)-f(\cdot+j)}_\infty \\
                                                    & = 2 \sum_{j \geq 0} \abs{\delta_{right}^2 G(t,j)} \Omega_\infty(j+1),
        \end{aligned}
    \]
    with $\Omega_\infty(j) = \norm{f(\cdot) - f(\cdot+j)}_\infty$. 
   % Now let's prove that $\norm{t^{1-\frac{\alpha}{2}} \partial_t e^{t\Delta_d}f}_{p,q}$ is finite. 
   Then, we have that
    \[
        \begin{aligned}
            \norm{t^{1-\frac{\alpha}{2}} \partial_t e^{t\Delta_d}f}_{p,q}   &   \leq 
                                                                                \left(
                                                                                    \int_0^1
                                                                                        \abs{
                                                                                            2 t^{1-\frac{\alpha}{2}} 
                                                                                            \sum_{ j \geq 0}
                                                                                                \abs{\delta^2_{right} G(t, j)} \Omega_p(j+1)
                                                                                        }^q
                                                                                        \frac{dt}{t}
                                                                                \right)^{\frac{1}{q}} \\
                                                                            &   + 
                                                                                \left(
                                                                                    \int_1^\infty
                                                                                        \abs{
                                                                                            2 t^{1-\frac{\alpha}{2}} 
                                                                                            \sum_{j \geq 0}
                                                                                                \abs{\delta^2_{right} G(t, j)} \Omega_p(j+1)
                                                                                        }^q
                                                                                        \frac{dt}{t}
                                                                                \right)^{\frac{1}{q}} \\
                                                                            &   =: A + B.
        \end{aligned}    
    \]
    By using \Cref{lemma:inequality_heat_kernel} with $\beta = -\frac{1}{2}$ we get that
    $\abs{\delta^2_{right} G(t,j)} \leq \dfrac{C}{(1+j^2)t^{\frac{1}{2}}}$, $j \geq 0,\,t>0$. Therefore,
    \[
        \begin{aligned}
            A   & \leq C \left(
                                \int_0^1
                                    \abs{
                                        t^{\frac{1}{2}-\frac{\alpha}{2}}
                                        \sum_{j \geq 1}
                                            \frac{\Omega_p(j)}{1+(j-1)^2}
                                    }^q
                                    \frac{dt}{t}
                            \right)^{\frac{1}{q}}.
        \end{aligned}    
    \]
    By taking into account that $\frac{1}{1+(j-1)^2} \leq \frac{2}{j^2}$, $j \geq 1$, and Minkowski's integral inequality, we obtain that
    \[
        A   \leq  C \left(
                            \int_{0}^1
                                \abs{
                                    t^{\frac{1}{2}-\frac{\alpha}{2}}
                                    \sum_{j \geq 1}
                                        \frac{\Omega_p(j)}{j^2}
                                }^q
                                \frac{dt}{t}
                        \right)^{\frac{1}{q}}
            \leq  C \sum_{j \geq 1}
                        \frac{\Omega_p(j)}{j^2}
                        \left(
                            \int_{0}^1
                                \abs{
                                    t^{\frac{1}{2}-\frac{\alpha}{2}}
                                }^q
                                \frac{dt}{t}
                        \right)^{\frac{1}{q}}
            \leq  C \sum_{j \geq 1}
                        \frac{\Omega_p(j)}{j^2}. 
    \]
    Notice that $2 =% \alpha + (1-\alpha) + 1$
    \alpha+1/q+2-\alpha-1/q$, so we can use Hölder's inequality to get 
    \[
        A  \leq  C \biggl(\sum_{j\geq 1} \biggl\| \frac{f(\cdot+j)-f(\cdot) }{j^{\alpha}} \biggr\|_p^q\frac{1}{j}\biggr)^{1/q}<\infty.
    \]
    Regarding B, we split it into 2 different, by integrals using Minkowski's inequality as follows
    \[
        \begin{aligned}
            B   & \leq  \left(
                            \int_1^\infty
                                \abs{
                                    2 t^{1-\frac{\alpha}{2}} 
                                    \sum_{ 0 \leq j \leq \sqrt{t}}
                                        \abs{\delta^2_{right} G(t, j)} \Omega_p(j+1)
                                }^q
                                \frac{dt}{t}
                        \right)^{\frac{1}{q}} \\
                &   +
                        \left(
                            \int_1^\infty
                                \abs{
                                    2 t^{1-\frac{\alpha}{2}} 
                                    \sum_{j > \sqrt{t}}
                                        \abs{\delta^2_{right} G(t, j)} \Omega_p(j+1)
                                }^q
                                \frac{dt}{t}
                        \right)^{\frac{1}{q}} \\
                &   =: B1 + B2.
        \end{aligned}    
    \]

    To compute the bound for $B1$, we apply \Cref{lemma:inequality_heat_kernel} with $\beta = -\frac{3}{2}$ so that
    $\abs{\delta^2_{right} G(t,j)} \leq \dfrac{C}{t^{\frac{3}{2}}}$, for $j \geq 0$ and $t>0$, to get
    $$  B1 \leq C \left(\int_1^\infty
                                \abs{
                                    t^{-\frac{\alpha}{2}-\frac{1}{2}} 
                                    \sum_{ 0 \leq j \leq \sqrt{t}}
                                        \Omega_p(j+1)
                                }^q
                                \frac{dt}{t}
                        \right)^{\frac{1}{q}}.$$
 
    Now we define the function $g := \sum_{j=0}^\infty \Omega_p(j+1)\chi_{[j+1,j+2)} $. Then, we can write \newline
    $\displaystyle\sum_{0\leq j \leq \sqrt{t}} \Omega_p(j+1)=\int_0^{\left[\sqrt{t}\right]+2} g(x) \, dx$.
    By using Hardy's inequality (see \Cref{Hardy_inequality}) one obtains that
    \[
        \begin{aligned}
            B1   & \leq C   \left(
                                \int_1^\infty
                                    \abs{
                                        t^{-\frac{\alpha}{2}-\frac{1}{2}}
                                        \int_0^{\left[\sqrt{t}\right]+2}
                                            g(x) \, dx
                                    }^q
                                    \frac{dt}{t}
                            \right)^{\frac{1}{q}}
                  \leq C   \left(
                                \int_1^\infty
                                    \abs{
                                        t^{-\frac{\alpha}{2}-\frac{1}{2}}
                                            \int_0^{3\sqrt{t}}
                                                g(x) \, dx
                                    }^q
                                    \frac{dt}{t}
                            \right)^{\frac{1}{q}} \\
                & = C      \left( \int_{3}^\infty \left(\int_0^{u} g(x) \, dx\right)^q u^{-q(\alpha+1)-1} \, du \right)^{\frac{1}{q}}
                  \leq C   \left( \int_{0}^\infty (x g(x))^q x^{-q(\alpha+1)-1} \, dx \right)^{\frac{1}{q}} \\
                & \leq C \left(\sum_{j=1}^\infty \frac{\norm{f(\cdot)-f(\cdot+j)}_p^q}{j^{q\alpha+1}}\right)^{\frac{1}{q}}
                  <\infty.         
        \end{aligned}
    \]

    On the other hand, by using \Cref*{lemma:inequality_heat_kernel} with $\beta = 0$, we have that
    $$B2\leq C \left(
                                \int_1^\infty
                                    \abs{
                                        t^{1-\frac{\alpha}{2}}
                                        \sum_{j > \sqrt{t}}
                                            \frac{\Omega_p(j+1)}{j^3}
                                        }^q 
                                    \frac{dt}{t}
                            \right)^{\frac{1}{q}}.$$  
                            
    We define the function $h := \sum _{j=0}^\infty \chi_{[j+1,j+2)} \frac{\Omega_p(j+1)}{j^3}$, so that we can write 
    $\sum_{ j>\sqrt{t}} \frac{\Omega_p(j+1)}{j^3}=\int_{\left[\sqrt{t}\right]+2}^\infty h(x)\,dx$.
    By using Hardy's inequality (\Cref{Hardy_inequality}) we have that
    \[
        \begin{aligned}
            B & \leq C \left(
                            \int_1^\infty
                                \abs{
                                    t^{1-\frac{\alpha}{2}}
                                    \int_{\left[\sqrt{t}\right]+2}^\infty
                                        h(x) \, dx
                                }^q
                                \frac{dt}{t}
                        \right)^{\frac{1}{q}}
                \leq C \left(
                            \int_1^\infty
                                \abs{
                                    t^{1-\frac{\alpha}{2}}
                                    \int_{\sqrt{t}}^\infty h(x) \, dx
                                }^q
                                \frac{dt}{t}
                        \right)^{\frac{1}{q}} \\
              & =    C \left(\int_1^\infty \left(\int_u^{\infty} h(x) \, dx\right)^q u^{q(2-\alpha)-1} \, du \right)^{\frac{1}{q}} 
                \leq C \left(\int_{0}^\infty (x h(x))^q x^{q(2-\alpha)-1} \, dx \right)^{\frac{1}{q}} \\ 
              & \leq C \left(\sum_{j=1}^\infty \frac{\norm{f(\cdot)-f(\cdot+j)}_p^q}{j^{q\alpha+1}}\right)^{\frac{1}{q}}
                <\infty.  
        \end{aligned}    
    \]
    Therefore, we have proved that $f \in \Lambda_H^{\alpha,p,q}$, $1 \leq p \leq \infty$ and $1 \leq q < \infty$.

    Suppose now that $f \in C^{\alpha,p,\infty}(\ZZ)$, with $1\leq p\leq \infty$. From equation \eqref{4.1}, it follows that
    \[
        \begin{aligned}
            \norm{t^{1-\frac{\alpha}{2}} \partial_t e^{t\Delta_d}f}_{p,\infty}  & =     \sup_{t > 0}
                                            t^{1-\frac{\alpha}{2}}
                                                                                            \norm{\partial_t e^{t\Delta_d}f}_p
                                                                                  \leq  \sup_{t > 0}
                                                                                            t^{1-\frac{\alpha}{2}}
                                                                                            \sum_{j \in \ZZ}
                                                                                                \abs{\partial_t G(t,j)}
                                                                                                \norm{f(\cdot)-f(\cdot\!+\!j)}_p \\
                                                                                & \leq \biggl(\sup_{j\neq 0}\frac{ \norm{f(\cdot)-f(\cdot\!+\!j)}_p}{|j|^{\alpha}}\biggl)
                                                                                        \sup_{t > 0}
                                                                                            t^{1-\frac{\alpha}{2}}
                                                                                            \norm{\partial_t G(t,\cdot) \abs{\cdot}^\alpha}_1.  
        \end{aligned}    
    \]
    By using the fact that $\norm{\partial_t G(t,\cdot) \abs{\cdot}^\alpha}_1 \leq C t^{\frac{\alpha}{2}-1}$ 
    (see \cite[Proof of Theorem 3.3]{AbadiasDeLeonContreras2022}) we conclude that
    $\norm{t^{1-\frac{\alpha}{2}} \partial_t e^{t\Delta_d}f}_{p,\infty}<\infty$, so $f \in \Lambda_H^{\alpha,p,\infty}$.
 
\end{proof}

\begin{proposition}
    \label{prop:Lambda_H_implies_Lambda_P}
    Let $\alpha > 0$ and $f: \ZZ \to \RR$ such that $\sum_{j\in\ZZ} \frac{\abs{f(j)}}{1+\abs{j}^2} < \infty$.
    If $f \in \Lambda_H^{\alpha,p,q}$, $1 \leq p,q \leq \infty$, then $f \in \Lambda_P^{\alpha,p,q}$.
\end{proposition}
\begin{proof}
Let $f \in \Lambda_H^{\alpha,p,q}$, $1 \leq p,q \leq \infty$. By using Lemmata \ref{lemma:inequality_add_derivative} and \ref{lemma:inequality_remove_derivative} (see also \Cref{remark:amount_of_derivatives_on_t}), it is enough to prove that $\norm{y^{l-\alpha}\partial_y^l e^{-y\sqrt{-\Delta_d}f}}_{p,q}<\infty$  for $l$ the least even number such that $l>[\alpha]+1$ and $l/2>[\alpha/2]+1.$
 %   By hypotheses, we know that $\sum_{j\in\ZZ} \frac{\abs{f(j)}}{1+\abs{j}^2} < \infty$, so we just have to prove that 
  %  $\norm{y^{l-\alpha}\partial_y^l e^{-y\sqrt{-\Delta_d}f}}_{p,q}$ is finite for $l=[\alpha]+1.$

 %   Let $k=\left[\frac{l+1}{2}\right]$. Next inequality follows applying \Cref{lemma:inequality_add_derivative}(ii) if $l$ is even, and it is an equality if $l$ is odd. 
 %   So we get in general ($l=[\alpha]+1\in\NN$)
%    \[
   %     \norm{y^{l-\alpha}\partial_y^l e^{-y\sqrt{-\Delta_d}}f}_{p,q} \leq 
        %\max\left\{1,\frac{1}{\alpha}\right\} \norm{y^{2k-\alpha}\partial_y^{2k} e^{-y\sqrt{-\Delta_d}}f}_{p,q}.   
 %   \]
  %  Also by \Cref{remark:some_semigroup_properties_work_fine} and 
Let $l$ be the even number satisfying the conditions above and suppose that $1\leq p\leq\infty$ and $1 \leq q < \infty$. From \Cref{lemma:inequalities_integral} we have that
    \[
        \norm{\partial_y^{l} e^{-y\sqrt{-\Delta_d}}f}_p \leq  \int_0^\infty 
                                                                \frac{1}{2 \pi}
                                                                \frac{y e^{ \frac{-y^2}{4t}}}{t^{\frac{3}{2}}}
                                                                \norm{\partial_t^{l/2} e^{t\Delta_d}f}_p
                                                            \, dt,
    \]
    and
  
            $$\norm{y^{l-\alpha}\partial_y^l e^{-y\sqrt{-\Delta_d}}f}_{p,q} \leq  C_\alpha\left(
           \int_0^\infty\left(
                                                                                                y^{l-\alpha}
                                                                                                    \int_0^\infty
                                                                                                    \frac{1}{2\pi}
                                                                                                    \frac{y e^{ \frac{-y^2}{4t}}}{t^{\frac{3}{2}}}
                                                                                                    \norm{\partial_t^{l/2} e^{t\Delta_d}f}_p
                                                                                                    \, dt
                                                                                            \right)^q
                                                                                        \frac{dy}{y}
                                                                                    \right)^{\frac{1}{q}},
        $$
    where $C_\alpha$ is a positive constant depending on $\alpha$. 
    %Notice that  $\left(\frac{y^2}{t}\right)^{\frac{3}{2}} e^{ \frac{-y^2}{4t}}$ as a function of $t\in (0,\infty)$ reaches its maximum value at $t=\frac{y^2}{6}$.
   % Thus, following the ideas appearing in \cite[Theorem 4.1]{DeLeonContrerasTorrea2020} and \cite[Theorem 5.6]{DeLeonContrerasTorrea2021}, one gets
   Notice that for every $\gamma\geq 0$ there is $C>0$ such that $\left(\frac{y}{t^{1/2}}\right)^{\gamma} e^{ \frac{-y^2}{4t}}\le C$, for every $t,y>0.$ So 
    \[
        \begin{aligned}
            \norm{y^{l-\alpha}\partial_y^le^{-y\sqrt{-\Delta_d}}f}_{p,q}    & \leq  C_\alpha
                                                                                    \left[
                                                                                        \left(
                                                                                            \int_0^\infty
                                                                                                \left(
                                                                                                    y^{l-2-\alpha}
                                                                                                        \int_0^{y^2}
                                                                                                        \norm{\partial_t^{l/2} e^{t\Delta_d}f}_p
                                                                                                        \, dt
                                                                                                \right)^q
                                                                                            \frac{dy}{y}
                                                                                        \right)^{\frac{1}{q}}
                                                                                    \right. \\
                                                                            & +     \left.
                                                                                           \left(
                                                                                                \int_0^\infty
                                                                                \left(y^{l-\alpha}\int_{y^2}^\infty\frac{\norm{\partial_t^{l/2}  e^{t\Delta_d}f       }_p
                                                                                                            }{
                                                                                                                t
                                                                                                            }
                                                                                                            \, dt
                                                                                                    \right)^q
                                                                                                \frac{dy}{y}
                                                                                            \right)^{\frac{1}{q}}
                                                                                    \right].
        \end{aligned}
    \]
    For both integrals we perform the change of variables ($y=\sqrt{s}$) and we use Hardy's inequality (\Cref{Hardy_inequality}), which
    yields to
    \[
        \norm{y^{l-\alpha}\partial_y^le^{-y\sqrt{-\Delta_d}}f}_{p,q} \leq C_\alpha \norm{t^{l/2-\frac{\alpha}{2}} \partial_t^{l/2} e^{t\Delta_d}f}_{p,q}<\infty.
    \]
   % which is bounded because $f \in \Lambda_H^{\alpha,p,q}$. 
    The case $q=\infty$ follows the same steps.
\end{proof}

\begin{proposition}
    \label{prop:Lambda_P_implies_C_case_0_1}
    Let $0 < \alpha < 1$, $1\leq p,q \leq \infty$ and $f \in \Lambda_P^{\alpha,p,q}$. Then, $f \in C^{\alpha,p,q}(\ZZ)$.
\end{proposition}
\begin{proof}
    Suppose that $1 \leq q < \infty$. Then, we can write%e can rewrite the series that appears in the definition of $C^{\alpha,p,q}(\ZZ)$ as
    \[
        \sum_{j\neq 0} \frac{\norm{f(\cdot+j)-f(\cdot)}_p^q}{\abs{j}^{\alpha q +1}} 
        = 2 \sum_{j=1}^\infty \frac{\norm{f(\cdot+j)-f(\cdot)}_p^q}{j^{\alpha q +1}} 
        = 2\int_1^\infty \frac{\norm{f(\cdot+[t])-f(\cdot)}_p^q}{[t]^{\alpha q +1}} \, dt.
    \]
    Since for $t\geq 1$ it holds that $[t]\geq \frac{t}{2}$, then
    \[
        \sum_{j\neq 0} \frac{\norm{f(\cdot+j)-f(\cdot)}_p^q}{\abs{j}^{\alpha q +1}} 
        \leq C_{\alpha,q} \int_0^\infty (t^{-\alpha}\omega(t,p))^q \frac{dt}{t},
    \]
    where $\displaystyle\omega(t,p) := \sup_{\substack{0 \leq j \leq t\\
    j\in\NN_0}} \norm{f(\cdot+j)-f(\cdot)}_p$.

    Moreover, for every $j \leq t $, we have that
    \[
        \begin{aligned}
            \norm{f(\cdot+j)-f(\cdot)}_p    & \leq  2 \norm{f - e^{-t\sqrt{-\Delta_d}}f}_p 
                                                    + \norm{e^{-t\sqrt{-\Delta_d}}f(\cdot)-e^{-t\sqrt{-\Delta_d}}f(\cdot+j)}_p \\
                                            & \leq  2 \int_0^t \norm{\partial_y e^{-y\sqrt{-\Delta_d}}f}_p \, dy 
                                                    + t \norm{\delta_{right}e^{-t\sqrt{-\Delta_d}}f}_p,
        \end{aligned}  
    \]
     where we have used the fact that $f$ satisfying $\sum_{j\in\ZZ}\frac{|f(j)|}{1+|j|^2}<\infty$ %\in  \Lambda_P^{\alpha,p,q}$ 
     implies that $\lim_{y\to 0}e^{-y\sqrt{-\Delta_d}}f(n)=f(n)$, for $n\in\ZZ$, (see \cite[Lemma 2.12.B]{AbadiasDeLeonContreras2022}) and  either Minkowski's integral inequality (when $1\leq p< \infty$) or a straightforward inequality (when $p=\infty$).
    Thus,
    \[
        \begin{aligned}
            \left(
                \int_0^\infty
                    \left(
                        t^{-\alpha}\omega(t,p)
                    \right)^q
                    \frac{dt}{t}
            \right)^{\frac{1}{q}}               & \leq 2\left(
                                                            \int_0^\infty
                                                                \left(
                                                                    t^{-\alpha}
                                                                    \int_0^t
                                                                        \norm{\partial_y e^{-y\sqrt{-\Delta_d}}f}_p
                                                                    \, dy
                                                                \right)^q
                                                                \frac{dt}{t}
                                                        \right)^{\frac{1}{q}} \\
                                                & +     \left(
                                                            \int_0^\infty
                                                                \left(
                                                                    t^{1-\alpha} \norm{\delta_{right} e^{-t\sqrt{-\Delta_d}}f}_p
                                                                \right)^q
                                                                \frac{dt}{t}
                                                        \right)^{\frac{1}{q}} \\
                                                & =: I + II.
        \end{aligned}
    \]
    On the one hand, by  using Hardy's inequality (\Cref{Hardy_inequality}), we obtain that
    \[
        I \leq C \left( \int_0^\infty \left(y^{1-\alpha} \norm{\partial_y e^{-y\sqrt{-\Delta_d}}f}_p \right)^q \frac{dy}{y}\right)^{\frac{1}{q}}<\infty.  
    \]
    On the other hand, by using \Cref{lemma:inequality_interchange_derivative}(ii),
    \[
        II 
        \leq
        C\norm{y^{1-\alpha} \partial_y e^{-t\sqrt{-\Delta_d}}f}_{p,q}<\infty. 
    \]
    Therefore, the result is proved for $1\leq q < \infty$.  The proof for $q = \infty$ follows similarly.
    %from where 
   % We conclude that if $f \in \Lambda_P^{\alpha,p,q}$ then $f \in C^{\alpha,p,q}(\ZZ)$ for $1\leq q < \infty$.
   % The proof for $q = \infty$ follows similarly.
\end{proof}

\begin{theorem}
    \label{theorem:equal_spaces_0_1}
    Let $0 < \alpha <1$ and $1\leq p,q \leq \infty$. It holds that $ C^{\alpha,p,q}(\ZZ) = \Lambda_H^{\alpha,p,q} = \Lambda_P^{\alpha,p,q}$. 
\end{theorem}
\begin{proof}
   From \Cref{prop:C_implies_Lambda_H_case_0_1} and    \Cref{prop:Lambda_P_implies_C_case_0_1}  we know that $ C^{\alpha,p,q}(\ZZ) \subset \Lambda_H^{\alpha,p,q}$ and $ \Lambda_P^{\alpha,p,q} \subset C^{\alpha,p,q}(\ZZ)$. It remains to prove that 
    $ \Lambda_H^{\alpha,p,q} \subset \Lambda_P^{\alpha,p,q}.$ 

    In virtue of \Cref{prop:Lambda_H_implies_Lambda_P}, it suffices to check that if  $\frac{f}{1+|\cdot|^\alpha} \in \ell^q(\ZZ, \mu)$ then 
    $\sum_{j\in\ZZ} \frac{\abs{f(j)}}{1+\abs{j}^2} < \infty$.  Indeed, if $\frac{f}{1+|\cdot|^\alpha} \in \ell^1(\ZZ, \mu)$  we have that
    \[
        \begin{aligned}
            \sum_{j\in \ZZ} \frac{\abs{f(j)}}{1+\abs{j}^2}  & =     \abs{f(0)} + \sum_{j\neq 0} \frac{\abs{f(j)}}{1+\abs{j}^\alpha}
                                                                    \frac{1}{\abs{j}} \left( \frac{(1+\abs{j}^\alpha)\abs{j}}{1+\abs{j}^2} \right) \\
                                                            & \leq  \abs{f(0)} + C \sum_{j \neq 0} \frac{\abs{f(j)}}{1+\abs{j}^\alpha}
                                                                    \frac{1}{\abs{j}} < \infty.      
        \end{aligned}
    \]
    Now suppose that ${\frac{f}{1+|\cdot|^\alpha}} \in \ell^q(\ZZ, \mu)$, with   $1 <q < \infty$. By taking $\gamma = \frac{\alpha q + 1}{2q} \in \left( \frac{1}{2},1 \right)$, we have that
    \[
        \begin{aligned}
            \sum_{j\in \ZZ} \frac{\abs{f(j)}}{1+\abs{j}^2}  & =     \sum_{j\in \ZZ} 
                                                                        \frac{\abs{f(j)}}{(1+\abs{j}^2)^\gamma}
                                                                        \frac{1}{(1+\abs{j}^2)^{1-\gamma}} \\
                                                            & \leq  \left(
                                                                        \sum_{j\in \ZZ} \frac{\abs{f(j)}^q}{(1+\abs{j}^2)^{\gamma q}}
                                                                    \right)^{\frac{1}{q}}
                                                                    \left(
                                                                        \sum_{j\in \ZZ} \frac{1}{(1+\abs{j}^2)^{(1-\gamma) q'}}
                                                                    \right)^{\frac{1}{q'}}
                                                              =: A \cdot B,   
        \end{aligned}
    \]
    where $q'$ is the conjugate exponent of $q$. Observe that $A$ is clearly finite and $B$ is also finite because
    $2(1-\gamma)q'>1$. Lastly, we know that if $\frac{f}{1+\abs{\cdot}^\alpha}\in \ell^\infty(\ZZ)$, then it
    follows that  $\sum_{j\in \ZZ} \frac{\abs{f(j)}}{1+\abs{j}^2}<\infty$. 
\end{proof}

\subsection{Case \texorpdfstring{$0 < \alpha < 2$}{0 < α < 2}}
\begin{proposition}
    \label{prop:Lambda_P_implies_C2_case_0_2}
    Let $0 < \alpha < 2$ and $1\leq p,q \leq \infty$. If $f \in \Lambda_P^{\alpha,p,q},$ then 
    \[
\sum_{j\neq 0} \biggl\| \frac{ f(\cdot-j)- 2f(\cdot) + f(\cdot+j) }{|j|^{\alpha}} \biggr\|_p^q\frac{1}{|j|}<\infty,\quad \text{if }1\leq q<\infty,
    \]
    and \[
\sup_{j\neq 0} \biggl\| \frac{ f(\cdot-j)- 2f(\cdot) + f(\cdot+j) }{|j|^{\alpha}} \biggr\|_p<\infty,\quad \text{if } q=\infty.
\]
\end{proposition}
\begin{proof}
    Let $f \in \Lambda_P^{\alpha,p,q}$. If $0 < \alpha < 1$ and $1\leq q<\infty,$  from \Cref{prop:Lambda_P_implies_C_case_0_1} we have that $f \in C^{\alpha,p,q}(\ZZ),$ so
    \[
        \biggl(\sum_{j\neq 0} \biggl\| \frac{ f(\cdot-j)- 2f(\cdot) + f(\cdot+j) }{|j|^{\alpha}} \biggr\|_p^q\frac{1}{|j|}\biggr)^q \leq 2  \biggl(\sum_{j\neq 0} \biggl\| \frac{ f(\cdot+j)- f(\cdot)}{|j|^{\alpha}} \biggr\|_p^q\frac{1}{|j|}\biggr)^q<\infty.
    \]
The case  $0 < \alpha < 1$ and $q=\infty$ is analogous.

    Assume that $1 \leq \alpha < 2$ and $1 \leq q < \infty$. We can rewrite the sum similarly as we did in the proof of 
    \Cref{prop:Lambda_P_implies_C_case_0_1} to get 
    \[
        \biggl(\sum_{j\neq 0} \biggl\| \frac{ f(\cdot-j)- 2f(\cdot) + f(\cdot+j) }{|j|^{\alpha}} \biggr\|_p^q\frac{1}{|j|}\biggr)^q \leq   C_{\alpha,q}
                                                                    \left(
                                                                        \int_0^\infty (t^{-\alpha}\omega(t,p))^q \frac{dt}{t}
                                                                    \right)^{\frac{1}{q}},
    \]
    where $\omega(t,p) := \sup_{\substack{0 \leq j \leq t\\j\in\NN_0}}\norm{f(\cdot+j)-2f(\cdot)+f(\cdot-j)}_p$.
   Let  $j\in\NN$ such that $j \leq t$. We can write
    \[
        \begin{aligned}
            & \norm{f(\cdot+j)-2f(\cdot)+f(\cdot-j)}_p \\ 
            & \leq    \norm{
                            f(\cdot + j) - e^{-t\sqrt{-\Delta_d}}f(\cdot+j)
                            -2 {(}f(\cdot) - e^{-t\sqrt{-\Delta_d}}f(\cdot){)}
                            + f(\cdot - j) - e^{-t\sqrt{-\Delta_d}}f(\cdot-j) 
                        }_p \\
            &  \quad + \norm{
                            e^{-t\sqrt{-\Delta_d}} \left( f(\cdot+j)-2f(\cdot)+f(\cdot-j) \right)
                        }_p  = I + II.
        \end{aligned}   
    \]
   By using Minkowski's integral inequality, \cite[Lemma 2.12.B]{AbadiasDeLeonContreras2022}    and the fact that
    $\partial_u e^{-u\sqrt{-\Delta_d}}f = - \int_u^t \partial_w^2 e^{-w\sqrt{-\Delta_d}}f \, dw + \partial_t e^{-t\sqrt{-\Delta_d}}f$, we get that
    \begin{align*}
        I &\leq \int_0^t \norm{\partial_u e^{-u\sqrt{-\Delta_d}} \left( f(\cdot+j)-2f(\cdot)+f(\cdot-j) \right) }_p \, du\\
        &\leq \int_0^t \int_u^t \norm{\partial_w^2 e^{-w\sqrt{-\Delta_d}} \left( f(\cdot+j)-2f(\cdot)+f(\cdot-j) \right) }_p \, dw \, du \\
                & \quad + \int_0^t \norm{\partial_t e^{-t\sqrt{-\Delta_d}} \left( f(\cdot+j)-2f(\cdot)+f(\cdot-j) \right) }_p \, du, \\
                & = I_1 + I_2.
        \end{align*}  
    On the one hand, we have that
    \[
        \begin{aligned}
            I_1  & \leq  4 \int_0^t \int_u^t \norm{\partial_w^2 e^{-w\sqrt{-\Delta_d}} f }_p \, dw \, du
                  =     4 \int_0^t \int_0^w \norm{\partial_w^2 e^{-w\sqrt{-\Delta_d}} f }_p \, du \, dw \\
                & =     4 \int_0^t w \norm{\partial_w^2 e^{-w\sqrt{-\Delta_d}} f }_p \, dw.    
        \end{aligned}
    \]
    On the other hand, we can write% for $j\geq 1$ we have %(the case $j=0$ implies $I_2=0$)
    \[
        \begin{aligned}
            I_2  & \leq  2 t \norm{\partial_t e^{-t\sqrt{-\Delta_d}} \left( f(\cdot+j)-f(\cdot) \right) }_p \\
                &  =    2 t \norm{\sum_{j'=1}^{j} \partial_t e^{-t\sqrt{-\Delta_d}} \left( f(\cdot+j')-f(\cdot+j'-1) \right) }_p \\
                & \leq  2 t^2 \norm{\partial_t \delta_{right} e^{-t\sqrt{-\Delta_d}} f }_p.
        \end{aligned}
    \]
    Moreover, since $j\le t,$
    %Now we compute the value of $II$% for $j\geq 1$ %(the case $j=0$ implies $II=0$),
    \[
    \begin{aligned}
            II  & =\norm{\sum_{j'=1}^{j} \delta_{right} e^{-t\sqrt{-\Delta_d}} (f(\cdot-j')-f(\cdot+j'-1)) }_p \\
                & = \norm{\sum_{j'=1}^{j} \sum_{k=0}^{2j'-2} \delta_{right}^2 e^{-t\sqrt{-\Delta_d}} f(\cdot-j'+k) }_p 
                  \leq \sum_{j'=1}^{j} \sum_{k=0}^{2j'-2} \norm{\delta_{right}^2 e^{-t\sqrt{-\Delta_d}} f}_p \\
                & = j^2 \norm{\delta_{right}^2 e^{-t\sqrt{-\Delta_d}} f}_p \leq  t^2 \norm{\delta_{right}^2 e^{-t\sqrt{-\Delta_d}} f}_p.
        \end{aligned}
    \]
    Therefore, we have got the following inequality %for $0\leq j\leq t,$
    \[
        \begin{aligned}
            \norm{f(\cdot+j)-2f(\cdot)+f(\cdot-j)}_p    & \leq 4 \int_0^t w \norm{\partial_w^2 e^{-w\sqrt{-\Delta_d}} f }_p \, du 
                                                             + 2 t^2 \norm{\partial_t \delta_{right} e^{-t\sqrt{-\Delta_d}} f }_p \\
                                                        & \quad + t^2 \norm{\delta_{right}^2 e^{-t\sqrt{-\Delta_d}} f}_p,
        \end{aligned}    
    \]
so
 \[
        \begin{aligned}
            \left(
                \int_0^\infty
                    \left(
                        t^{-\alpha}\omega(t,p)
                    \right)^q
                    \frac{dt}{t}
            \right)^{\frac{1}{q}}               & \leq 4 \left(
                \int_0^\infty
                    \left(
                        t^{-\alpha} \int_0^t w \norm{\partial_w^2 e^{-w\sqrt{-\Delta_d}} f }_pdw
                    \right)^q
                    \frac{dt}{t}
            \right)^{\frac{1}{q}} \\
                                                & +    2\left(
                \int_0^\infty
                    \left(
                        t^{2-\alpha}  \norm{\partial_t \delta_{right} e^{-t\sqrt{-\Delta_d}} f }_p
                    \right)^q
                    \frac{dt}{t}
            \right)^{\frac{1}{q}} \\
                                                & +    \left(
                \int_0^\infty
                    \left(
                        t^{2-\alpha}  \norm{\delta_{right}^2 e^{-t\sqrt{-\Delta_d}} f}_p
                    \right)^q
                    \frac{dt}{t}
            \right)^{\frac{1}{q}}.
        \end{aligned}
    \]
By applying Hardy's inequality (\Cref{Hardy_inequality}) in the first summand of above expression, and \Cref{lemma:inequality_interchange_derivative}(ii) in the second and third ones, we get the result.

The case  $1 \leq \alpha < 2$ and $q=\infty$ can be proved by following the same ideas.
    
\end{proof}

\begin{proposition}
    \label{prop:C2_implies_Lambda_H_case_0_2}
    Let $0 < \alpha < 2$, $1\leq p,q \leq \infty$ and $f:\ZZ \to \RR$ such that
    $\frac{f}{1+|\cdot|^\alpha} \in \ell^q(\ZZ, \mu)$ with \[
\sum_{j\neq 0} \biggl\| \frac{ f(\cdot-j)- 2f(\cdot) + f(\cdot+j) }{|j|^{\alpha}} \biggr\|_p^q\frac{1}{|j|}<\infty,\quad\text{if } 1\leq q<\infty,
    \] and \[
\sup_{j\neq 0} \biggl\| \frac{ f(\cdot-j)- 2f(\cdot) + f(\cdot+j) }{|j|^{\alpha}} \biggr\|_p<\infty, \quad\text{if } q=\infty.
\] Then $f \in \Lambda_H^{\alpha,p,q}$.
\end{proposition}
\begin{proof}
    Since for every  $t>0$ it holds that $G(t,j)=G(t,-j)$,  $j \in \NN$, and $\partial_t e^{t \Delta_d}1 = 0$ (where $1$ denotes in this case the sequence with all entries equal to 1),  we have that 
    \[
        \begin{aligned}
            \norm{\partial_t e^{t\Delta_d}f}_p  & = \norm{\frac{1}{2} \sum_{j\in \ZZ} \partial_t G(t,j) (f(n-j)-2f(n)+f(n+j))}_p\\
                                                & \leq \sum_{j \geq 0} \abs{\delta_{right}^2 G(t,j)} \Omega_p(j+1),
        \end{aligned}
    \]
    with $\Omega_p(j) = \norm{f(\cdot-j) -2f(\cdot) + f(\cdot+j)}_p$. %and where we have used the same techniques as in the proof of   \Cref{prop:C_implies_Lambda_H_case_0_1}. 
    
    The rest of the proof  follows from the same techniques used  to prove \Cref{prop:C_implies_Lambda_H_case_0_1}.% with the definition of $\Omega_p(j)$ adapted, we get the result.
\end{proof}

From \Cref{prop:Lambda_H_implies_Lambda_P,prop:Lambda_P_implies_C2_case_0_2,prop:C2_implies_Lambda_H_case_0_2}, it follows next theorem, which is one of our main results: the complete characterization of the discrete Besov and Zygmund spaces for $0<\alpha<2.$

 \begin{theorem}
    \label{theorem:equal_spaces_0_2}
    Let $0< \alpha < 2$, $1 \leq p,q \leq \infty$ and $f : \ZZ \to \RR$ be a function such that 
    $\frac{f}{1+|\cdot|^\alpha} \in \ell^q(\ZZ, \mu)$ and $\sum_{j\in\ZZ} \frac{\abs{f(j)}}{1+\abs{j}^2} < \infty$.
    The following are equivalent:
    \begin{enumerate}
        \item $f \in \Lambda_H^{\alpha,p,q}$.
        
        \item $f \in \Lambda_P^{\alpha,p,q}$.
        
        \item $f$ satisfies
        \[
\sum_{j\neq 0} \biggl\| \frac{ f(\cdot-j)- 2f(\cdot) + f(\cdot+j) }{|j|^{\alpha}} \biggr\|_p^q\frac{1}{|j|}<\infty,\quad\text{if } 1\leq q<\infty,
    \] and \[
\sup_{j\neq 0} \biggl\| \frac{ f(\cdot-j)- 2f(\cdot) + f(\cdot+j) }{|j|^{\alpha}} \biggr\|_p<\infty, \quad\text{if } q=\infty.
\]
    \end{enumerate}
\end{theorem}
\begin{remark}
    { Observe that the assumption $\sum_{j\in\ZZ} \frac{\abs{f(j)}}{1+\abs{j}^2} < \infty$ in the previous result is only needed for $1\le \alpha<2$. Indeed, we have proved the previous result for $0<\alpha<1$ in \Cref{theorem:equal_spaces_0_1} without that assumption (in this case it is deduced from the hypothesis $\frac{f}{1+|\cdot|^\alpha} \in \ell^q(\ZZ, \mu)$). In addition,  by proceeding in an analogous way as in the proof of \Cref{prop:Lambda_P_implies_C2_case_0_2} (but performing the change of variables $\tilde{t}=\sqrt{t}$ in $\left(\int_0^\infty (t^{-\alpha}\omega(t,p))^q \frac{dt}{t}\right)^{\frac{1}{q}}$) we can prove that $f \in \Lambda_H^{\alpha,p,q}\implies f$ satisfies \textit{(3)}, so  it can be proved that   \emph{(1)} and \emph{(3)} are equivalent for $0< \alpha < 2$, without  imposing the assumption $\sum_{j\in\ZZ} \frac{\abs{f(j)}}{1+\abs{j}^2} < \infty$.}
\end{remark}
  %  The proof of of \Cref{theorem:equal_spaces_0_2} follows from 

\subsection{General case}

\begin{theorem}
    \label{theorem:step_down_up_lambda_H_space}
    Let $ \alpha > 1$, $1 \leq p,q \leq \infty$ and $f : \ZZ \to \RR$. Then, $f \in \Lambda_H^{\alpha,p,q}$ if, and only if
    $\delta_{right}f \in \Lambda_H^{\alpha-1,p,q}$.
\end{theorem}
\begin{proof}

    Let $f\in \Lambda_H^{\alpha,p,q}$. First, we shall prove that 
    $\frac{{\delta_{right}f}}{1+\abs{\cdot}^{\alpha-1}}\in \ell^q(\ZZ,\mu)$.
Observe that when $q=\infty$, in virtue of the embedding $\ell^p(\ZZ) \hookrightarrow \ell^\infty(\ZZ)$, we
    have that $f \in \Lambda_H^{\alpha,\infty,\infty}$, so  $\frac{{\delta_{right}f}}{1+\abs{\cdot}^{\alpha-1}}\in \ell^\infty(\ZZ)$ (see the proof of \cite[Theorem 3.6]{AbadiasDeLeonContreras2022}).

    Let now $ 1 \leq q < \infty.$ We want to prove that $\norm{\frac{\delta_{right}f}{1+\abs{\cdot}^{\alpha-1}}}_{q,\mu}<\infty$. Since $\frac{{f}}{1+\abs{\cdot}^\alpha}\in \ell^q(\ZZ,\mu)$, then
    $\frac{{f}}{1+\abs{\cdot}^{\alpha+\frac{1}{q}}}\in \ell^\infty(\ZZ),$ so by
    \Cref{remark:some_semigroup_properties_work_fine} and \cite[Lemma 2.12.A.(iii)]{AbadiasDeLeonContreras2022} we have $\lim_{t\to 0} e^{t\Delta_d} f(n)=f(n)$, for every $n\in \ZZ$. Then,
    \[
        \begin{aligned}
            \norm{
                \frac{\delta_{right}f}{1+\abs{\cdot}^{\alpha-1}}
                }_{q,\mu}                  \leq  &  | \delta_{right}f(0) | +  \left(\sum_{n\neq 0}
                                                                    \left(
                                                            \frac{\abs{\delta_{right}f(n)}}{1+\abs{n}^{\alpha-1}}
                                                     \right)^q
                                                     \frac{1}{\abs{n}}
                                                    \right)^{\frac{1}{q}} \\
                                                               \leq &  | \delta_{right}f(0) | +  \left(
                                                                        \sum_{n\neq 0}
                                                                            \left(
                                                                                \sup_{0<t<n^2}
                                                                                \frac{\abs{e^{t\Delta_d}\delta_{right}f(n)}}
                                                                                    {1+\abs{n}^{\alpha-1}}
                                                                            \right)^q
                                                                            \frac{1}{\abs{n}}
                                                                        \right)^{\frac{1}{q}} \\
                                                               \leq &  | \delta_{right}f(0) | +  \left(
                                                                        \sum_{n\neq 0}
                                                                            \left(
                                                                                \sup_{0<t<n^2}
                                                                                \frac{
                                                                                    \abs{
                                                                                        e^{t\Delta_d}\delta_{right}f(n)
                                                                                            -e^{n^2\Delta_d}\delta_{right}f(n)
                                                                                    }
                                                                                }
                                                                                {1+\abs{n}^{\alpha-1}}
                                                                            \right)^q
                                                                            \frac{1}{\abs{n}}
                                                                        \right)^{\frac{1}{q}} \\
                                                                    & + \left(
                                                                        \sum_{n\neq 0}
                                                                            \left(
                                                                                \frac{\abs{e^{n^2\Delta_d}\delta_{right}f(n)}}
                                                                                    {1+\abs{n}^{\alpha-1}}
                                                                            \right)^q
                                                                            \frac{1}{\abs{n}}
                                                                        \right)^{\frac{1}{q}} \\
                                                                    & =:  | \delta_{right}f(0) | +A + B.
        \end{aligned}
    \]
    Since $\frac{{f}}{1+\abs{\cdot}^\alpha}\in \ell^q(\ZZ,\mu),$ from \Cref{lemma:inequality_norms_n_square_semigroup} we deduce that $B$ is finite.

    Now we study the sum $A$. Suppose that $\alpha \in (1,3)$. Then,
    \[
        \begin{aligned}
            A   & \leq  \left(
                            \sum_{n \neq 0}
                            \left(\sup_{0<t<n^2}
                                \int_t^{n^2}
                                    \abs{\partial_u\delta_{right} e^{u \Delta_d}f(n)} \, du
                            \right)^q
                            \abs{n}^{-q(\alpha-1)-1}
                        \right)^{\frac{1}{q}} \\
            & \leq  \left(
                            \sum_{n \neq 0}
                            \left(
                                \int_0^{n^2}
                                    \abs{\partial_u\delta_{right} e^{u \Delta_d}f(n)} \, du
                            \right)^q
                            \abs{n}^{-q(\alpha-1)-1}
                        \right)^{\frac{1}{q}} \\
                & \leq  \left(
                            \sum_{n \neq 0}
                            \left(
                                \int_0^{n^2}
                                    \norm{\partial_u\delta_{right} e^{u \Delta_d}f}_p \, du
                            \right)^q
                            \abs{n}^{-q(\alpha-1)-1}
                        \right)^{\frac{1}{q}} \\
                & \leq  C%2^{\frac{1}{q}}
                        \left(
                            \int_1^\infty
                            \left(
                                \int_0^{[t]^2}
                                    \norm{\partial_u\delta_{right} e^{u \Delta_d}f}_p \, du
                            \right)^q
                            [t]^{-q(\alpha-1)-1}
                            \, dt
                        \right)^{\frac{1}{q}}.
        \end{aligned}
    \]
    Since $t>1$, we have that $[t] > \frac{t}{2}$ and %because $\alpha >1$ we have that
    \[
        \begin{aligned}
            A   & \leq  C%2^{\alpha-1+\frac{2}{q}}
                        \left(
                            \int_1^\infty
                            \left(
                                \int_0^{t^2}
                                    \norm{\partial_u\delta_{right} e^{u \Delta_d}f}_p \, du
                            \right)^q
                            t^{-q(\alpha-1)-1}
                            \, dt
                        \right)^{\frac{1}{q}} \\
                & =     C
                        \left(
                            \int_1^\infty
                            \left(
                                \int_0^{x}
                                    \norm{\partial_u\delta_{right} e^{u \Delta_d}f}_p \, du
                            \right)^q
                            x^{\frac{-q(\alpha-1)}{2}-1}
                            \, dx
                        \right)^{\frac{1}{q}}.
        \end{aligned}
    \]
    By using Hardy's inequality (see \Cref{Hardy_inequality}) we get that
    \[
        A   \leq  C
                \left(
                    \int_0^\infty
                    \left(
                        u \norm{\partial_u\delta_{right} e^{u \Delta_d}f}_p
                    \right)^q
                    u^{\frac{-q(\alpha-1)}{2}-1}
                \right)^{\frac{1}{q}}
            =     C \norm{u^{\frac{3}{2}-\frac{\alpha}{2}} \partial_u \delta_{right} e^{u\Delta_d} f}_{p,q}.
    \]
    Thus,  we use \Cref*{lemma:inequality_interchange_derivative} and Remark \ref{remark:amount_of_derivatives_on_t} to 
obtain 
    \[
        A   \leq C \norm{u^{\frac{3}{2}-\frac{\alpha}{2}} \partial_u \delta_{right} e^{u\Delta_d} f}_{p,q}
            \leq C \norm{u^{2-\frac{\alpha}{2}} \partial_u^2 e^{u\Delta_d} f}_{p,q} < \infty.
    \]
  \begin{comment}  Thus, in the particular case  of $\alpha \in (1,2)$, we use \red{\Cref{lemma:inequality_remove_derivative}} %\Cref*{lemma:inequality_add_derivative} 
    to get
    \[
        A   \leq C \norm{u^{\frac{3}{2}-\frac{\alpha}{2}} \partial_u \delta_{right} e^{u\Delta_d} f}_{p,q}
            \leq C \norm{u^{1-\frac{\alpha}{2}} \partial_u e^{u\Delta_d} f}_{p,q} < \infty,
    \]
    \blue{while} if $\alpha \in (2,3)$, we use \Cref*{lemma:inequality_interchange_derivative} to 
obtain \Marta{en verdad este caso sirve para todo $\alpha\in (1,3)$ debido al {\Cref{lemma:inequality_remove_derivative}}}
    \[
        A   \leq C \norm{u^{\frac{3}{2}-\frac{\alpha}{2}} \partial_u \delta_{right} e^{u\Delta_d} f}_{p,q}
            \leq C \norm{u^{2-\frac{\alpha}{2}} \partial_u^2 e^{u\Delta_d} f}_{p,q} < \infty.
    \]
\end{comment}
    Now consider $\alpha \in [3,5)$. Observe that the techniques that we will present in this part of the proof are enough to prove all the cases
    $\alpha \in [2k+1,2k+3)$ with $k \in \NN$, but we just prove the case $\alpha \in [3,5)$. We start by using the fact that
    the semigroup is the solution to the heat equation, and splitting our sum into 2 different parts, 
    \[
        \begin{aligned}
            A   & \leq  \left(
                            \sum_{n \neq 0}
                            \left(
                                \int_0^{n^2}
                                    \abs{e^{u \Delta_d} \delta_{right}^3 f(n-1)} \, du
                            \right)^q
                            \abs{n}^{-q(\alpha-1)-1}
                        \right)^{\frac{1}{q}} \\
                & \leq  \left(
                            \sum_{n \neq 0}
                            \left(
                                \int_0^{n^2}
                                    \int_u^{n^2}
                                        \abs{\partial_w e^{w \Delta_d} \delta_{right}^3 f(n-1)}
                                        \, dw
                                    \, du
                            \right)^q
                            \abs{n}^{-q(\alpha-1)-1}
                        \right)^{\frac{1}{q}} \\
                & \qquad + \left(
                            \sum_{n \neq 0}
                            \left(
                                \frac{\abs{e^{n^2 \Delta_d} \delta_{right}^3 f(n-1)}}{\abs{n}^{\alpha-3}}
                            \right)^q
                            \frac{1}{\abs{n}}
                        \right)^{\frac{1}{q}}\\
                & =:    A_1 + A_2.
        \end{aligned}
    \]
    Since $\frac{f}{1+|\cdot|^{\alpha}}\in \ell^q(\ZZ,\mu)$ we have that $\frac{f(\cdot -1)}{1+|\cdot|^{\alpha}}\ell^q(\ZZ,\mu),$ so by \Cref*{lemma:inequality_norms_n_square_semigroup} we get that $A_2$ is finite. For $A_1$ we interchange the integrals, apply
    Hardy's inequality and the techniques for the case $\alpha \in (1,3)$, to get
    \[
        \begin{aligned}
            A_1 & \leq C \left(
                            \sum_{n \neq 0}
                            \left(
                                \int_0^{n^2}
                                    \int_u^{n^2}
                                        \norm{\partial_w^2 \delta_{right} e^{w \Delta_d}  f}_p
                                        \, dw
                                    \, du
                            \right)^q
                            \abs{n}^{-q(\alpha-1)-1}
                        \right)^{\frac{1}{q}}\\
            & \leq  C
                        \left(
                            \int_1^\infty
                            \left(
                                \int_0^{t^2}
                                    w \norm{ \partial_w^2 \delta_{right} e^{w \Delta_d} f}_p \, dw
                            \right)^q
                            t^{-q(\alpha-1)-1} \, dt
                        \right)^{\frac{1}{q}} \\
                & =     C
                        \left(
                            \int_1^\infty
                            \left(
                                \int_0^{x}
                                    w \norm{ \partial_w^2 \delta_{right} e^{w \Delta_d} f}_p \, dw
                            \right)^q
                            x^{-\frac{q(\alpha-1)}{2}-1} \, dx
                        \right)^{\frac{1}{q}} \\
                & \leq  C
                        \left(
                            \int_0^\infty
                            \left(
                                w^{\frac{5}{2}-\frac{\alpha}{2}}
                                \norm{ \partial_w^2 \delta_{right} e^{w \Delta_d} f}_p
                            \right)^q dw                      \right)^{\frac{1}{q}} \\
                          %  w^{-\frac{q(\alpha-1)}{2}-1} \,
                &= C\norm{w^{\frac{5}{2}-\frac{\alpha}{2}} \partial_w^2 \delta_{right} e^{w \Delta_d} f}_{p,q}.
        \end{aligned} 
    \]
\begin{comment}
    
   If $\alpha \in [3,4)$ we use \Cref{lemma:inequality_remove_derivative}, so \Marta{mismo caso que antes, sirve para todo $\alpha\in [3,5)$ debido al {\Cref{lemma:inequality_remove_derivative}}}
    $A_1 \leq C \norm{w^{2-\frac{\alpha}{2}} \partial_w^2  e^{w \Delta_d} f}_{p,q} < \infty,$ while if $\alpha \in [4,5)$ we use \end{comment} 
   
    From   \Cref{lemma:inequality_interchange_derivative}  and Remark \ref{remark:amount_of_derivatives_on_t} (in case $\alpha\in [3,4)$) we deduce that 
    $$A_1 \leq C \norm{w^{3-\frac{\alpha}{2}} \partial_w^3 e^{w \Delta_d} f}_{p,q} < \infty.$$
    Finally, we prove the semigroup condition. First we shall consider  $\alpha \in (1,3)$. \begin{comment} 
    $\alpha \in (1,2)$. By using \Cref{lemma:inequality_remove_derivative}(i)
    %and the fact that $f \in \Lambda_H^{\alpha,p,q}$,
    \[
        \norm{t^{1-\frac{\alpha-1}{2}}\partial_t e^{t \Delta_d} \delta_{right} f}_{p,q}
            = \norm{t^{\frac{3}{2}-\frac{\alpha}{2}}\partial_t \delta_{right} e^{t \Delta_d} f}_{p,q}
            \leq C \norm{t^{1-\frac{\alpha}{2}}\partial_t e^{t \Delta_d} f}_{p,q} < \infty.
    \]
    %For the case $\alpha \in [2,3)$ we proceed similarly but in this case we use 
    \end{comment}
    By using \Cref{lemma:inequality_interchange_derivative}(i) and Remark \ref{remark:amount_of_derivatives_on_t} (in case $\alpha\in (1,2)$) we get that
    \[
        \norm{t^{1-\frac{\alpha-1}{2}}\partial_t e^{t \Delta_d} \delta_{right} f}_{p,q}
            = \norm{t^{\frac{3}{2}-\frac{\alpha}{2}}\partial_t \delta_{right} e^{t \Delta_d} f}_{p,q}
            \leq C \norm{t^{2-\frac{\alpha}{2}}\partial_t^2 e^{t \Delta_d} f}_{p,q} < \infty.
    \]
    In general, \begin{comment}if we consider $\alpha \in [2k+1,2k+2)$ with $k\in \NN,$ by applying \Cref{lemma:inequality_remove_derivative}(i) we have that 
    \Marta{ mismo comentario que antes}
    \begin{align*}
        \norm{t^{k+1-\frac{\alpha-1}{2}}\partial_t^{\red{k+1}} e^{t \Delta_d} \delta_{right} f}_{p,q}
            &= \norm{t^{k+ \frac{3}{2}-\frac{\alpha}{2}}\partial_t^{\red{k+1}}  \delta_{right} e^{t \Delta_d} f}_{p,q}
            \\&\leq C \norm{t^{k+1-\frac{\alpha}{2}}\partial_t^{\red{k+1}}  e^{t \Delta_d} f}_{p,q} < \infty,
    \end{align*}  and \end{comment}
 if $\alpha \in [2k+1,2k+3)$ with $k\in \NN,$ by using \Cref*{lemma:inequality_interchange_derivative}(i) and Remark \ref{remark:amount_of_derivatives_on_t} (in case $\alpha \in [2k+1,2k+2)$) we get that
    \begin{align*}
        \norm{t^{k+1-\frac{\alpha-1}{2}}\partial_t^{k+1}  e^{t \Delta_d} \delta_{right} f}_{p,q}
            &= \norm{t^{k+\frac{3}{2}-\frac{\alpha}{2}}\partial_t^{k+1}  \delta_{right} e^{t \Delta_d} f}_{p,q}
            \\&\leq C \norm{t^{k+2-\frac{\alpha}{2}}\partial_t^{k+2}  e^{t \Delta_d} f}_{p,q} < \infty.
\end{align*}

    Assume now that $\delta_{right} f \in \Lambda_H^{\alpha-1,p,q}$. By definition, we have that
    $\frac{{\delta_{right}f}}{1+\abs{\cdot}^{\alpha-1}} \in \ell^q(\ZZ, \mu)$. Thus, the proof of \Cref*{lemma:C_space_bound_growth}
    gives that
    \begin{align*}
            \norm{\frac{f}{1+\abs{\cdot}^\alpha}}_{q,\mu} &\leq |f(0)| +C+ C\left(
                                \sum_{j=1}^\infty
                                \left(\abs{\delta_{right}f(j-1)} \right)^q \frac{1}{j^{ (\alpha-1) q +1}}
                            \right)^{\frac{1}{q}} \\
                    &+    C\left(
                                \sum_{j=1}^\infty
                                \left(\abs{\delta_{right}f(-j)} \right)^q \frac{1}{j^{ (\alpha-1) q +1}}
                            \right)^{\frac{1}{q}}<\infty, 
    \end{align*}
 so $\frac{f}{1+\abs{\cdot}^{\alpha}} \in \ell^q(\ZZ, \mu)$. Suppose that $\alpha \in [2k+1,2k+3)$ with $k \in \NN_{0}$ (being $\alpha \neq 1$). By
     using \Cref{lemma:inequality_interchange_derivative}(i) and \Cref{remark:amount_of_derivatives_on_t} (in case $\alpha \in [2k+1,2k+2$)) we have that
    \[
        \begin{aligned}
            \norm{t^{k+2-\frac{\alpha}{2}}\partial_t^{k+2} e^{t \Delta_d}f}_{p,q} 
                & =     \norm{t^{k+\frac{3}{2}-\frac{\alpha}{2}}\partial_t^{k+1} \delta_{right} e^{t \Delta_d} f }_{p,q} \\
                & \leq  C \norm{t^{k+1-\frac{\alpha-1}{2}}\partial_t^{k+1} e^{t \Delta_d} \delta_{right}f }_{p,q} < \infty.
        \end{aligned}
    \]
    %%%VERSION ANTERIOR AQUI ABAJO
 \begin{comment}   Suppose that $\alpha \in [2k+1,2k+2)$ with $k \in \NN_{0}$ (being $\alpha \neq 1$). From \Cref*{lemma:inequality_interchange_derivative}(i) we get that
    \[
        \begin{aligned}
            \norm{t^{k+1-\frac{\alpha}{2}}\partial_t^{k+1} e^{t \Delta_d}f}_{p,q} 
                & =     \norm{t^{k+1-\frac{\alpha}{2}}\partial_t^{k} \delta_{right} e^{t \Delta_d} \delta_{right}f }_{p,q} \\
                & \leq  C \norm{t^{k+1-\frac{\alpha-1}{2}} \partial_t^{k+1} e^{t \Delta_d} \delta_{right}f }_{p,q} < \infty.
        \end{aligned}
    \]
   % which is finite because $\delta_{right} f \in \Lambda_H^{\alpha-1,p,q}$. 
    Analogously, if $\alpha \in [2k,2k+1)$ with $k \in \NN$,
    by using \Cref*{lemma:inequality_remove_derivative}(i) we have that
    \[
        \begin{aligned}
            \norm{t^{k+1-\frac{\alpha}{2}}\partial_t^{k+1} e^{t \Delta_d}f}_{p,q} 
                & =     \norm{t^{k+1-\frac{\alpha}{2}}\partial_t^{k} \delta_{right} e^{t \Delta_d} \delta_{right}f }_{p,q} \\
                & \leq  C \norm{t^{k-\frac{\alpha-1}{2}}\partial_t^{k} e^{t \Delta_d} \delta_{right}f }_{p,q} < \infty.
        \end{aligned}
    \]
    \end{comment}
   We conclude that $f \in \Lambda_H^{\alpha,p,q}$.
\end{proof}

\begin{theorem}
    \label{theorem:step_down_lambda_P_space}
    Let $ \alpha > 1$, $1 \leq p,q \leq \infty$ and $f \in \Lambda_P^{\alpha,p,q}.$ Then,
    $\delta_{right}f \in \Lambda_P^{\alpha-1,p,q}$.
\end{theorem}
\begin{proof}
    Since $f \in \Lambda_P^{\alpha,p,q}$ we have that $\displaystyle\sum_{n \in \ZZ} \frac{\abs{f(n)}}{1+\abs{n}^2}<\infty$, so it is clear that 
    $\displaystyle\sum_{n \in \ZZ} \frac{\abs{\delta_{right} f(n)}}{1+\abs{n}^2}<\infty$.

    Let $l_1 = [\alpha] + 1$ and $l_2 =[\alpha-1]+1= [\alpha]$. Then, by using \Cref*{lemma:inequality_interchange_derivative}(ii) we have that
    \[
        \begin{aligned}
            \norm{y^{l_2-(\alpha-1)} \partial_y^{l_2} e^{-y\sqrt{-\Delta_d}} \delta_{right} f}_{p,q} 
                & =     \norm{y^{l_2-(\alpha-1)} \partial_y^{l_2} \delta_{right} e^{-y\sqrt{-\Delta_d}} f}_{p,q} \\
                & \leq  C \norm{y^{l_1-\alpha} \partial_y^{l_1} e^{-y\sqrt{-\Delta_d}} f}_{p,q} < \infty.
        \end{aligned}
    \]
   We conclude that $\delta_{right}f \in \Lambda_P^{\alpha-1,p,q}$.
\end{proof}

Finally, we can prove our main theorem.\\

{\it Proof of \Cref{theorem:equal_spaces_general_case}.}
    We prove first \emph{(A1)}. In \Cref*{theorem:equal_spaces_0_1} we have proved the result for $0 < \alpha < 1$. Let $k < \alpha < k+1$,
    for certain $k \in \NN$, and  assume that $f \in \Lambda_H^{\alpha,p,q}$. Then, by applying $k$ times \Cref*{theorem:step_down_up_lambda_H_space}
    we get that $\delta_{right}^k f \in \Lambda_H^{\alpha-k,p,q}$ and in virtue of \Cref*{theorem:equal_spaces_0_1} and the definition of $C^{\alpha-k,p,q}$
    we get that $f \in C^{\alpha,p,q}(\ZZ)$.

    Conversely, suppose that $f \in C^{\alpha,p,q}(\ZZ)$, $\alpha>1$, $\alpha\not\in\NN$. Then, from \Cref{lemma:C_space_bound_growth} we know that $\frac{{f}}{1+\abs{\cdot}^\alpha} \in \ell^q(\ZZ, \mu)$.
     Moreover, the definition of the space gives that $\delta_{right}^k f \in C^{\alpha-k,p,q}$,
    and \Cref{theorem:equal_spaces_0_1} implies that $\delta_{right}^k f \in \Lambda_H^{\alpha-k,p,q}$.
    By applying $k$ times \Cref*{theorem:step_down_up_lambda_H_space} we conclude that $f \in \Lambda_H^{\alpha,p,q}$.

    Regarding the proof of \emph{(A2)}, we proceed as in the proof of \emph{(A1)} but we use \Cref{theorem:equal_spaces_0_2}
    (see \Cref{remark:some_semigroup_properties_work_fine}) instead of \Cref{theorem:equal_spaces_0_1}.

    In virtue of \Cref{prop:Lambda_H_implies_Lambda_P} and \emph{(A1)}, to establish \emph{(B)} we only need to prove that if
    $f \in \Lambda_P^{\alpha,p,q}$ then $f \in C^{\alpha,p,q}(\ZZ)$. Let $f \in \Lambda_P^{\alpha,p,q}$. By applying $k$ times
    \Cref{theorem:step_down_lambda_P_space} we get that $\delta_{right}^k f \in \Lambda_P^{\alpha-k,p,q}$ and from \Cref{theorem:equal_spaces_0_1}
    and the definition of $C^{\alpha-k,p,q}(\ZZ)$ we conclude that $f \in C^{\alpha,p,q}(\ZZ)$.

    Regarding the proof of \emph{(B2)}, we proceed as in the proof of \emph{(B1)} but we use \Cref{theorem:equal_spaces_0_2}
    instead of \Cref{theorem:equal_spaces_0_1}.
\edproof

\section{Applications}\label{Applications}

In this section, we shall prove regularity results for Bessel potentials and fractional powers of the discrete Laplacian in the Besov spaces defined through
the heat semigroup. To this aim, we recall the definition of the fractional powers of the discrete Laplacian by using the semigroup method,
see \cite{CiaurriRoncalStingaTorreaVarona2018}.

Let $I$ denote the identity operator. For good enough functions $f:\ZZ\to\RR$, we define the following operators:

\begin{itemize}
    \item The Bessel potential of order $\beta>0$,
            \[
                \left(I-\Delta_d\right)^{-\beta / 2} f(n) = 
                    \frac{1}{\Gamma(\beta / 2)} 
                    \int_0^{\infty} e^{-\tau(I-\Delta_d)} f(n) \tau^{\beta / 2} \frac{d \tau}{\tau}, \quad n \in \ZZ.
            \]
    \item The positive fractional power of the Laplacian,
            \[
                \left(-\Delta_d\right)^{\beta} f(n) = 
                    \frac{1}{c_\beta}
                    \int_0^{\infty} 
                        \left(e^{\tau \Delta_d}-I\right)^{[\beta]+1} f(n) \frac{d \tau}{\tau^{1+\beta}}, \quad n \in \ZZ, \quad \beta>0,
            \]
          where $c_\beta=\int_0^{\infty}\left(e^{-\tau}-1\right)^{[\beta]+1} \frac{d \tau}{\tau^{1+\beta}}$.
    \item The negative fractional power of the Laplacian,
            \[
                \left(-\Delta_d\right)^{-\beta} f(n) = 
                    \frac{1}{\Gamma(\beta)}
                    \int_0^{\infty}
                        e^{\tau \Delta_d} f(n) \frac{d \tau}{\tau^{1-\beta}}, \quad n \in \ZZ, \quad 0<\beta<\frac{1}{2}.
            \]
\end{itemize}

The previous formulae come from the following gamma formulae, see \cite{CiaurriRoncalStingaTorreaVarona2018},
\[
    \lambda^{-\beta} = \frac{1}{\Gamma(\beta)} \int_0^{\infty} e^{-\lambda t} t^\beta \frac{dt}{t},
    \quad \text { and }
    \lambda^\beta=\frac{1}{c_\beta} \int_0^{\infty}\left(e^{-\lambda t}-1\right)^{[\beta]+1} \frac{dt}{t^{1+\beta}},    
\]
where $\beta>0$ and $\lambda$ is a complex number with $\Re \lambda \geq 0$.

As it was shown in \cite[Theorem 1.2]{AbadiasDeLeonContreras2022}, Bessel potentials of order $\beta>0$ are well defined for
$f \in \Lambda_H^{\alpha,\infty,\infty}, \alpha>0$. However, the fractional powers of the Laplacian, $\left(-\Delta_d\right)^{\pm \beta}$,
are not well defined in general for $\Lambda_H^{\alpha,\infty,\infty}$ functions and an additional condition is needed.
In \cite{CiaurriRoncalStingaTorreaVarona2018}, the authors assumed that the functions belong to the space
\[
    \ell_{\pm \beta}:=\left\{u: \ZZ \to \RR: \sum_{m \in \ZZ} \frac{|u(m)|}{(1+|m|)^{1 \pm 2 \beta}}<\infty\right\},    
\]
in order to define $\left(-\Delta_d\right)^{\pm \beta} f$, where $0<\beta<1$ in the case of the positive powers and $0<\beta<1 / 2$
for the negative ones. The choice of these spaces is justified since the discrete kernel in the pointwise formula
\begin{equation}
    \label{eq:fractionalLaplacianConvolution}
    \left(-\Delta_d\right)^{\pm \beta} f(n)=\sum_{m \in \ZZ} K_{\pm \beta}(n-m) f(m), n \in \ZZ,
\end{equation}
satisfies $K_\beta(m) \sim \frac{1}{|m|^{1+2 \beta}}$, whenever $0<\beta<1$ and $K_{-\beta}(m) \sim \frac{1}{|m|^{1-2 \beta}}$,
for $0<$ $\beta<1 / 2$, see \cite{CiaurriRoncalStingaTorreaVarona2018}. Observe that the negative powers of the Laplacian are only
well defined for $0<\beta<1 / 2$, since the integral that defines it is not absolutely convergent for $\beta \geq 1 / 2$.

In this section, in order to study regularity properties for positive powers larger than 1, we proceed as in \cite{AbadiasDeLeonContreras2022},
by extending the $\ell_\beta$ spaces for $\beta>0$, and working with the following extended kernel,
\[
    K_\beta(n):= \begin{cases}
                    0, & |n|-\beta-1 \in \mathbb{N}_0, \\
                    \frac{(-1)^{|n|} \Gamma(2 \beta+1)}{\Gamma(1+\beta+|n|) \Gamma(1+\beta-|n|)}, & \text { otherwise. }
                \end{cases}    
\]
For any $\beta >0,$ the kernel satisfies the same asymptotic estimates, that is, $K_\beta(m) \sim \frac{1}{|m|^{1+2 \beta}}.$ Note that when $\beta \in \mathbb{N}_0$, then $K_\beta(n)=0$ for all $|n| \geq \beta+1$. In fact, in
\cite[Lemma 4.1]{AbadiasDeLeonContreras2022} it was proved that if $f \in \ell_\beta$ then $\left(-\Delta_d\right)^{\beta} f$ is
well defined for $\beta >0$ and the identity \eqref{eq:fractionalLaplacianConvolution} holds. The cases $-1/2<\beta<0$ and $0<\beta<1$ had been proved previously in \cite{CiaurriRoncalStingaTorreaVarona2018}.

Now, we prove our main results of this section.
\begin{comment}    
\begin{theorem}
    \label{theorem:BesselPotentials}
    Let $\alpha, \beta>0$, $1 \leq p,q \leq \infty$ and $f: \ZZ \to\RR$ such that $f \in \Lambda_H^{\alpha,p,q}$,
    then $(I - \Delta_d)^{-\beta / 2} f \in \Lambda_H^{\alpha+\beta,p,q}$.
\end{theorem}\end{comment}
\vspace{0.5 cm}

{\it Proof of Theorem \ref{theorem:BesselPotentials}.}
    Let $f \in \Lambda_H^{\alpha,p,q}$ and $k = \left[ \frac{\alpha+\beta}{2}\right] + 1$. Notice that since $f \in \Lambda_H^{\alpha,p,q}$
    we have that $\frac{f}{1+\abs{\cdot}^{\alpha+\frac{1}{q}}} \in \ell^\infty(\ZZ),$ so by \cite[Lemma 2.12]{AbadiasDeLeonContreras2022}
    we have that 
    \[
        \abs{\left( I - \Delta_d\right)^{-\beta / 2} f(n)} \leq
            C \int_0^\infty \! e^{-\tau} \left(1+\abs{n}^{\alpha+1/q} + \tau^{\frac{\alpha + 1/q}{2}}\right) \tau^{\beta/2} \frac{d \tau}{\tau}
            \leq C (1 + \abs{n}^{\alpha + 1/q}), \quad  n \in \ZZ.
    \]
    This proves that the Bessel potential is well-defined.
    Moreover, from  \Cref{lemma:decay_heat_semigroup} \textit{(1)}  and Minkowski's integral inequality we have
    \[
        \norm{\frac{(I-\Delta_d)^{-\beta/2} f}{1+\abs{\cdot}^\alpha}}_{q,\mu} \leq 
            C \int_0^\infty e^{-\tau} \norm{\frac{e^{\tau} \Delta_d f}{1+\abs{\cdot}^\alpha}}_{q,\mu} \tau^{\beta/2} \frac{d \tau}{\tau} \leq
            C \int_0^\infty e^{-\tau} (1+ \tau^{\frac{\alpha+1/q}{2}}) \tau^{\beta/2} \frac{d \tau}{\tau} < \infty.
    \]
Since $  \norm{\frac{(I-\Delta_d)^{-\beta/2} f}{1+\abs{\cdot}^\alpha}}_{q,\mu}<\infty  $ implies that   $\norm{\frac{(I-\Delta_d)^{-\beta/2} f}{1+\abs{\cdot}^{\alpha+\beta}}}_{q,\mu}<\infty$, the size condition is satisfied. %Notice that when $q=\infty,$ it was proved in \cite[Theorem 1.3]{AbadiasDeLeonContreras2022} that $\frac{(I-\Delta_d)^{-\beta/2} f}{1+\abs{\cdot}^{\alpha+\beta}}\in\ell^\infty(\ZZ).$

    Finally, we prove the semigroup condition. First note that
    \[
        \begin{aligned}
            \norm{\partial_t^k e^{t \Delta_d} (I-\Delta_d)^{-\beta / 2} f}_p & =   \left(
                                                                                        \sum_{n \in \ZZ}
                                                                                            \abs{
                                                                                                \frac{1}{\Gamma(\beta/2)}
                                                                                                \int_0^\infty
                                                                                                    e^{-\tau}
                                                                                                    \partial_t^k e^{t\Delta_d}(e^{\tau\Delta_d}f)(n)
                                                                                                    \tau^{\beta/2} \frac{d \tau}{\tau}
                                                                                            }^p
                                                                                    \right)^{\frac{1}{p}} \\
                                                                             & \leq \frac{1}{\Gamma(\beta/2)}
                                                                                    \int_0^\infty
                                                                                        e^{-\tau}
                                                                                        \norm{\partial_{w}^k e^{w\Delta_d}f\Big|_{w=t+\tau}}_p
                                                                                        \tau^{\beta/2} \frac{d \tau}{\tau} \\
                                                                             & \leq \frac{1}{\Gamma(\beta/2)}
                                                                                    \int_t^\infty
                                                                                        \norm{\partial_u^k e^{u \Delta_d}f}_p
                                                                                        (u-t)^{\beta/2} \frac{du}{u-t}, \quad \text{ for } 1 \leq p < \infty, 
        \end{aligned}
    \]
  and the same inequality holds for $p=\infty$ directly {(one can intertwine the operator $\partial_t^k e^{t \Delta_d}$ and the integral in previous estimates as it is shown in \cite[Proof of Theorem 1.2]{AbadiasDeLeonContreras2022}).}
    
    Now we use \Cref{Hardy_convolution_inequality} for $1 \leq q < \infty$ and Remark \ref{remark:amount_of_derivatives_on_t} to obtain
    \[
        \begin{aligned}
            \norm{t^{k-\frac{\alpha+\beta}{2}} \partial_t^k e^{t \Delta_d}(I-\Delta_d)^{-\beta/2}f}_{p,q}
                & \leq  C
                    \left( 
                        \int_0^\infty
                            \left(
                                t^{k-\frac{\alpha+\beta}{2}} \!
                                \int_t^\infty
                                    \norm{\partial_u^k e^{u \Delta_d}f}_p \!
                                     \frac{(u-t)^{\beta/2}}{u-t} du
                            \right)^q
                            \frac{dt}{t}
                    \right)^{\frac{1}{q}} \\
                & \leq C
                    \left( 
                        \int_0^\infty
                            \left( 
                                t^{k-\frac{\alpha}{2}}
                                \norm{\partial_t^k e^{t \Delta_d}f}_p
                            \right)^q
                            \frac{dt}{t}
                    \right)^{\frac{1}{q}} \\
                & = C  \norm{t^{k-\frac{\alpha}{2}} \partial_t^k e^{t \Delta_d}f }_{p,q}< \infty.
              %   \leq C \norm{t^{l-\frac{\alpha}{2}} \partial_t^l e^{t \Delta_d}f }_{p,q} < \infty        
        \end{aligned}
    \]
%    where in the last line we have used \Cref{lemma:inequality_remove_derivative} several times until we get the decay
   % $l = \left[ \frac{\alpha}{2}\right] +1$ that assure us that the norm is finite.

    The case $q= \infty$ follows analogously, that is,
    \[
        \begin{aligned}
            \norm{t^{k-\frac{\alpha+\beta}{2}} \partial_t^k e^{t \Delta_d}(I-\Delta_d)^{-\beta/2}f}_{p,\infty}
                & \leq  C
                    \sup_{t>0}
                        t^{k-\frac{\alpha+\beta}{2}}
                        \int_t^\infty
                            \norm{\partial_u^k e^{u \Delta_d}f}_p
                            \frac{(u-t)^{\beta/2}}{u-t} du \\
                & \leq C 
                    \sup_{t>0} 
                        t^{k-\frac{\alpha}{2}} \norm{\partial_t^k e^{t \Delta_d}f}_p
                        \int_0^1 (1-y)^{\frac{\beta}{2}-1} y^{k-\frac{\alpha+\beta}{2}-1} \, dy \\
                & \leq C \sup_{t>0} t^{k-\frac{\alpha}{2}} \norm{\partial_t^k e^{t \Delta_d}f}_p
                  = C \norm{t^{k-\frac{\alpha}{2}} \partial_t^k e^{t \Delta_d}f}_{p,\infty} < \infty. %\\
              %  & \leq C \norm{t^{l-\frac{\alpha}{2}} \partial_t^l e^{t \Delta_d}f}_{p,\infty} < \infty      
        \end{aligned}
    \]
    %where, as before we have used \Cref{lemma:inequality_remove_derivative} several times and $l = \left[ \frac{\alpha}{2}\right] +1$.
\edproof

\begin{comment}\begin{theorem}
    \label{theorem:SchauderEstimates}
    Let $\alpha>0$, $0<\beta<1 / 2$, $1 \leq p,q \leq \infty$ and $f: \ZZ \to \RR$ such that $f \in \Lambda_H^{\alpha,p,q}\cap\ell_{-\beta}$, 
    then $\left(-\Delta_d\right)^{-\beta} f \in \Lambda_H^{\alpha+2 \beta,p,q}$.
\end{theorem}\end{comment}\\
\vspace{0.5 cm}

{\it Proof of Theorem \ref{theorem:SchauderEstimates}.}
    Notice that $f \in \ell_{-\beta}$ implies that $\left(-\Delta_d\right)^{-\beta} f $ is well defined and \eqref{eq:fractionalLaplacianConvolution}
    holds (see \cite{CiaurriRoncalStingaTorreaVarona2018}). Now we check that $\frac{\left(-\Delta_d\right)^{-\beta} f}{1+\abs{\cdot}^{\alpha+2\beta}} \in \ell^q\left(\ZZ, \mu\right)$.

    If $q=\infty$, then $f \in \Lambda_H^{\alpha,p,\infty}\subset \Lambda_H^{\alpha,\infty,\infty}$ % by the natural embedding of the $\ell^p(\ZZ)$ spaces we have that
 %   $f \in \Lambda_H^{\alpha,\infty,\infty}$ and in this case the result was proved in 
 and therefore $\frac{\left(-\Delta_d\right)^{-\beta} f}{1+\abs{\cdot}^{\alpha+2\beta}} \in \ell^\infty\left({\ZZ}\right)$, see \cite{AbadiasDeLeonContreras2022}.

    Let $1 \leq q < \infty$. In virtue of\eqref{eq:fractionalLaplacianConvolution}, we have to prove that 
    \[
        S:= \left(
                \sum_{n \in \ZZ} \abs{\sum_{j \in \ZZ} \frac{K_{-\beta}(n-j)f(j)}{1+\abs{n}^{\alpha+2\beta}}}^q \frac{1}{1+\abs{n}}
            \right)^{\frac{1}{q}}   <\infty.
    \]
    By using the bounds for $K_{-\beta}$ and the inverse triangle inequality we can split our sum into 3 different parts,
    \[
        \begin{aligned}
            S & \leq C  \left(
                        \sum_{n = 1}^\infty 
                            \left(
                                \sum_{j =1}^\infty
                                    \frac{\abs{f(j)}}{n^{\alpha+2\beta}\left(1+\abs{n-j}^{1-2\beta}\right)}
                            \right)^q
                            \frac{1}{n}
                        \right)^{\frac{1}{q}}
                +   C   \left(
                            \sum_{n = 1}^\infty 
                                \left(
                                    \frac{\abs{f(0)}}{n^{\alpha+1}}
                                \right)^q
                                \frac{1}{n}
                        \right)^{\frac{1}{q}} \\
              & +   C   \left(
                            \sum_{n = 1}^\infty 
                                \left(
                                    \sum_{j = 1}^\infty
                                        \frac{\abs{f(-j)}}{n^{\alpha+2\beta}\left(1+\abs{n-j}^{1-2\beta}\right)}
                                \right)^q
                                \frac{1}{n}
                        \right)^{\frac{1}{q}}+C\sum_{j\in\ZZ}\frac{|f(j)|}{1+|j|^{1-2\beta}}.
                \\
                &= S_1+ S_2 + S_3+S_4.
        \end{aligned}  
    \]
It is clear that $S_2,S_4 < \infty$.   We will estimate $S_1$ and $S_3$  together  by naming as  $a_j$ both
    $|f(j)|$ and $|f(-j)|$. On the one hand, $\frac{f}{1+\abs{\cdot}^\alpha} \in \ell^q(\ZZ,\mu)$ implies that $\frac{f}{1+\abs{\cdot}^{\alpha+1/q}} \in \ell^q(\ZZ)$. Thus, by {using the $\ell^q(\NN_0)$-boundedness of the discrete Ces\`aro operator, see \cite[Theorem 7.2]{AbadiasMiana}, for $1<q<\infty$ we get
    \begin{align*}
        \left(
            \sum_{n = 1}^\infty 
                \left(
                    \sum_{j =1}^{n-1}
                        \frac{a_j}{n^{\alpha+2\beta}(n-j)^{1-2\beta}}
                \right)^q
                \frac{1}{n}
            \right)^{\frac{1}{q}}&\le \left(
            \sum_{n = 1}^\infty 
                \left(\frac{1}{n^{2\beta}}
                    \sum_{j =1}^{n-1}
                        \frac{a_j}{j^{\alpha+1/q}(n-j)^{1-2\beta}}
                \right)^q
            \right)^{\frac{1}{q}}
       \\& \leq C \norm{ \frac{a_j}{|\cdot|^{\alpha+1/q}}}_q< \infty.
    \end{align*}}
  %  because $\frac{f}{1+\abs{\cdot}^\alpha} \in \ell^q(\ZZ,\mu)$. 
    For the case $q=1$ we use Tonelli's theorem to obtain
    \[
        \begin{aligned}
            \sum_{n = 1}^\infty 
                \sum_{j =1}^{n-1}
                    \frac{a_j}{n^{\alpha+2\beta+1}(n-j)^{1-2\beta}}
            & = 
            \sum_{j = 1}^\infty 
                a_j
                \sum_{n =j+1}^{\infty}
                    \frac{1}{n^{\alpha+2\beta+1}(n-j)^{1-2\beta}} \\
            %& \leq\sum_{j = 1}^\infty a_j \int_j^\infty         \frac{dx}{x^{\alpha+2\beta+1}(x-j)^{1-2\beta}}
            &\le \sum_{j = 1}^\infty\frac{a_j}{j^{\alpha+2\beta+1}}\sum_{n=j+1}^{2j}\frac{1}{(n-j)^{1-2\beta}}+\sum_{j = 1}^\infty\frac{a_j}{j^{1-2\beta}}\sum_{n=2j+1}^{\infty}\frac{1}{n^{\alpha+2\beta+1}}\\
            &\le \sum_{j = 1}^\infty\frac{a_j}{j^{\alpha+2\beta+1}}\int_{0}^{j}\frac{dx}{x^{1-2\beta}}+C<\infty,
          %&= C \sum_{j = 1}^\infty \frac{a_j}{j^{\alpha+1}},
        \end{aligned}
    \]
    where we have used that $\frac{f}{1+\abs{\cdot}^\alpha} \in \ell^1(\ZZ,\mu)$ and {$f\in \ell_{-\beta}$.}

%    If we consider in $S_1$ or $S_3$ the summand $j=n$, we observe that the resulting series is convergent because  $\frac{f}{1+\abs{\cdot}^\alpha} \in \ell^q(\ZZ,\mu)$.
    On the other hand, by {using the $\ell^q(\NN_0)$-boundedness of the discrete adjoint Ces\`aro operator, see \cite[Theorem 7.2]{AbadiasMiana}, for $1<q<\infty$, we get}
    \[
        \begin{aligned}
            \left(
                \sum_{n = 1}^\infty 
                    \left(
                        \sum_{j =n}^{2n}
                            \frac{a_j }{n^{\alpha+2\beta}(1+(j-n)^{1-2\beta})}
                    \right)^q
                    \frac{1}{n}
            \right)^{\frac{1}{q}}
            & \leq \left(\sum_{n = 1}^\infty 
                \left(\frac{a_n}{n^{\alpha+2\beta+1/q}}\right)^q\right)^{1/q}\\&+C  {\left(
                        \sum_{n = 1}^\infty 
                            \left(
                                \sum_{j =n+1}^{2n}
                                    \frac{a_j}{j^{\alpha+1/q}}\frac{1}{j^{2\beta}(j-n)^{1-2\beta}}
                            \right)^q
                    \right)^{\frac{1}{q}} }\\
            &{\le C \norm{ \frac{a_j}{|\cdot|^{\alpha+1/q}}}_q< \infty.}
        \end{aligned}
    \]
   % where we have used the same inequality as before, so we have to treat the case $q=1$ separately. 
   
 \begin{comment}   \[
        \begin{aligned}
            \left(
                \sum_{n = 1}^\infty 
                    \left(
                        \sum_{j =n+1}^{2n}
                            \frac{a_j (j-n)^{2\beta}}{n^{\alpha+2\beta}(j-n)}
                    \right)^q
                    \frac{1}{n}
            \right)^{\frac{1}{q}}
            & \leq C  \left(
                        \sum_{n = 1}^\infty 
                            \left(
                                \sum_{j =n+1}^{2n}
                                    \frac{a_j}{j^\alpha+2\beta+1} \frac{(j-n)^{2\beta}}{(j-n)}
                            \right)^q
                            n^{q-1}
                    \right)^{\frac{1}{q}} \\
            & \leq C \sum_{n=1}^\infty \left( \frac{a_j}{j^\alpha}\right)^q \frac{1}{j} < \infty,    
        \end{aligned}
    \]
    where we have used the same inequality as before, so we have to treat the case $q=1$ separately. 
    For this case we use Tonelli's theorem again so,\end{comment}
     When $q=1$ we use again  Tonelli's theorem so that
    \[
        \begin{aligned}
            \sum_{n = 1}^\infty 
                \sum_{j = n}^{2n}
                    \frac{a_j}{n^{\alpha+2\beta+1}(1+(j-n)^{1-2\beta})}
            & = \sum_{n = 1}^\infty 
                \frac{a_n}{n^{\alpha+2\beta+1}}+
            \sum_{j = 2}^\infty 
                a_j
                \sum_{ \frac{j}{2} \leq n  \leq j-1}
                    \frac{1}{n^{\alpha+2\beta+1}(j-n)^{1-2\beta}} \\
            & \leq C+C
            \sum_{j = 2}^\infty 
                \frac{a_j}{j^{\alpha+1}}
                \sum_{1 \leq m \leq \frac{j}{2}}
                    \frac{1}{m^{1-2\beta}(j-m)^{2\beta}}\\
                    &\le C+C\sum_{j = 2}^\infty 
                \frac{a_j}{j^{\alpha+2\beta+1}}
                \sum_{1 \leq m \leq \frac{j}{2}}
                    \frac{1}{m^{1-2\beta}}\\
                    &\le C+C\sum_{j = 2}^\infty 
                \frac{a_j}{j^{\alpha+2\beta+1}}
                \int_{0}^{\frac{j}{2}}
                    \frac{dx}{x^{1-2\beta}}<\infty,
        \end{aligned}
    \] where we have used that $\frac{f}{1+\abs{\cdot}^\alpha} \in \ell^1(\ZZ,\mu)$.

   Finally, observe that if $j \geq 2n+1$ we have that $\frac{1}{j-n} < \frac{2}{j}$, so by using the fact that $a_j \in \ell_{-\beta}$ we get
    \[
        \left(
            \sum_{n = 1}^\infty 
                \left(
                    \sum_{j =2n+1}^{\infty}
                        \frac{a_j}{1+(j-n)^{1-2\beta}}
                \right)^q
                \frac{1}{n^{q(\alpha+2\beta)+1}}
        \right)^{\frac{1}{q}}
        \leq C  \left(
                    \sum_{n = 1}^\infty 
                        \frac{1}{n^{q(\alpha+2\beta)+1}}
                \right)^{\frac{1}{q}}
        < \infty.
    \]

    It remains to check that $\norm{t^{k-\frac{\alpha+2\beta}{2}} \partial_t^k e^{t \Delta_d} (-\Delta_d)^{-\beta}f}_{p,q} < \infty$
    with $k = \left[ \frac{\alpha+2\beta}{2} \right]+ 1 $. %We start bound the $p$-norm as in the last proof.
    As in the proof of \Cref{theorem:BesselPotentials}, we observe that
    \[
        \begin{aligned}
            \norm{\partial_t^k e^{t \Delta_d} (-\Delta_d)^{-\beta}f}_{p} & = 
                \left(
                    \sum_{n \in \ZZ}
                        \abs{
                            \frac{1}{\Gamma(\beta)}
                            \int_{0}^\infty
                                \partial_w^k e^{w\Delta_d}f\Big|_{w=t + \tau}(n) 
                                \frac{d \tau}{\tau^{1-\beta}}
                        }^p
                \right)^{\frac{1}{p}} \\
            & \leq  \frac{1}{\Gamma(\beta)} 
                \int_t^\infty
                    \norm{\partial_t^k e^{u \Delta_d}f}_{p}
                    \frac{(u-t)^\beta}{(u-t)} du, \quad 1<q<\infty.
        \end{aligned}
    \]
    By using \Cref{Hardy_convolution_inequality} %, \Cref{lemma:inequality_remove_derivative}
    and \Cref{remark:amount_of_derivatives_on_t} we have that
    \[
        \begin{aligned}
            \norm{t^{k-\frac{\alpha+2\beta}{2}} \partial_t^k e^{t \Delta_d} (-\Delta_d)^{-\beta}f}_{p,q} 
                & \leq C \norm{t^{k-\frac{\alpha}{2}} \partial_t^k e^{t \Delta_d} f}_{p,q} < \infty.%\\
              %  & \leq C \norm{t^{l-\frac{\alpha}{2}} \partial_t^l e^{t \Delta_d} (-\Delta_d)^{-\beta}f}_{p,q} < \infty,
        \end{aligned}
    \]
 %   with $l = \left[ \frac{\alpha}{2}\right] +1$. 
    The case $q=\infty$ follows similarly.% ideas as in the proof of \Cref{theorem:BesselPotentials} and  this one.
\edproof
\begin{comment}
\begin{theorem}
    \label{theorem:HolderEstimates}
    Let $\alpha, \beta>0$, such that $0< 2 \beta<\alpha$, $1 \leq p,q \leq \infty$ and $f: \ZZ \to \RR$.
    \begin{enumerate}
        \item If $f \in \Lambda_H^{\alpha,p,q} \cap \ell_\beta$, then $\left(-\Delta_d\right)^\beta f \in \Lambda_H^{\alpha-2 \beta,p,q}$.
        \item If $\beta \in \NN$ and $f \in \Lambda_H^{\alpha,p,q}$, then 
            $\underbrace{(-\Delta_d) \circ (-\Delta_d) \circ \cdots \circ (-\Delta_d)}_{\beta \text{ times }}f \in \Lambda_H^{\alpha-2 \beta,p,q}$.
    \end{enumerate}
\end{theorem}\end{comment}
\vspace{0.5 cm}

{\it Proof of Theorem \ref{theorem:HolderEstimates}.}
    Notice that if $\beta \in \NN$ the integral definition of the fractional Laplacian coincides with the $\beta$ times composition of the operator
    (see \cite[Remark 4.2]{AbadiasDeLeonContreras2022}), and as $(-\Delta_d) f(n) = - \delta_{right}^{2} f(n-1)$, $n\in\ZZ,$ we can apply 
    \Cref{theorem:step_down_up_lambda_H_space} to prove epigraph $\emph{(2)}$ and epigraph $\emph{(1)}$  when $\beta \in \NN$.

    Now consider $\beta \notin \NN$ and  let $k = \left[ \frac{\alpha-2\beta}{2}\right] +1$. Observe that the study of the size condition follows the same steps as in the last proof but
    in this case with $2\beta$ instead of $-2\beta$. It remains to prove that
    $\norm{t^{k-\frac{\alpha-2\beta}{2}} \partial_t^k e^{t \Delta_d} (-\Delta_d)^\beta f}_{p,q}<\infty$.
  % We start bounding the following $p$-norm for
 If $1 \leq p < \infty $,
  by using Minkowski's integral inequality twice, we get
    \begin{align*}
        \norm{\partial_t^k e^{t\Delta_d}(-\Delta_d)^\beta f}_p  & =
      \norm{ \frac{1}{c_\beta}\partial_t^k e^{t\Delta_d}\left(\int_0^\infty\int_{[0,\tau]^l}
                                                                                            \partial_{\nu}^{l}  e^{\nu \Delta_d}f\Big|_{\nu=s_1+\dots+s_l}
                                                                                    d(s_1,\dots,s_l)
                                                                            \frac{d\tau}{\tau^{1+\beta} }\right)}_p\\
                                                                            &\leq    \frac{1}{c_\beta}
                                                                        \int_0^\infty
                                                                                \int_{[0,\tau]^l}
                                                                                        \norm{
                                                                                            \partial_{\nu}^{k+l}  e^{\nu \Delta_d}f\Big|_{\nu=t+s_1+\dots+s_l}
                                                                                        }_p
                                                                                    d(s_1,\dots,s_l)
                                                                            \frac{d\tau}{\tau^{1+\beta}},
    \end{align*}
    where $l = [\beta]+1$. In an analogous way the inequality  holds for $p=\infty$.
    Now we are going to focus on the integral over $[0,\tau]^l$. If $\beta \in (0,1)$, we apply Tonelli's theorem to obtain
    \[
        \begin{aligned}
            \frac{1}{c_{\beta}} 
            \int_0^\infty 
                \int_0^\tau  \norm{\partial_\nu^{k +1} e^{ \nu\Delta_d}f\Big|_{\nu=t+s_1}}_p ds_1 \frac{d\tau}{\tau^{1+\beta}}
                \begin{comment}
                    \norm{\partial_t^k \partial_{s_1} e^{(t+s_1) \Delta_d}f}_p ds_1 \frac{d\tau}{\tau^{1+\beta}}
                     \end{comment}
            & =
            \frac{1}{c_{\beta}} 
            \int_0^\infty  \norm{\partial_\nu^{k +1} e^{ \nu\Delta_d}f\Big|_{\nu=t+s_1}}_p
                \int_{s_1}^{\infty} \frac{d\tau}{\tau^{1+\beta}} ds_1 \\
            & = \frac{1}{c_{\beta}}
            \int_0^\infty \norm{\partial_\nu^{k +1} e^{ \nu\Delta_d}f\Big|_{\nu=t+s_1}}_p\frac{s_1^{1-\beta}}{s_1} \, ds_1\\
            & = \frac{1}{c_{\beta}}
            \int_t^\infty  \norm{\partial_u^{k +1} e^{ u\Delta_d}f}_p \frac{(u-t)^{1-\beta}}{u-t} \, du.
        \end{aligned}
    \]
    For $\beta > 1$, instead of integrating over $[0,\tau]^l$, we will compute the volume integral
    under the hyperplane which will be indeed
    bigger than our original integral, but easier to calculate. In order to do that we introduce some notation.
    We denote by $s = (s_1,s_2, \dots, s_l) \in \RR^{l}$ and by $s'= (s_2,\dots,s_l) \in \RR^{l-1}$ with $s_i \geq 0, \forall i=1,\dots,l$.
    We define the set $K_l(\theta)$ as $K_l(\theta) = \{ s \in \RR^{l} : 0 \leq s_1+\dots + s_l \leq \theta\}$. Hence, using Tonelli's theorem
    we have that
    \[
        \begin{aligned}
            \int_{[0,\tau]^l} \!\!
         \norm{  \partial_{\nu}^{k+l}  e^{\nu \Delta_d}f\Big|_{\nu=t+s_1+\dots+s_l}}_p \hspace{-0.15cm}ds&
       \leq 
        \int_{s \in K_l(l \tau)} 
             \norm{  \partial_{\nu}^{k+l}  e^{\nu \Delta_d}f\Big|_{\nu=t+s_1+\dots+s_l}}_p ds \\
        & = \hspace{-0.1cm}\int_{s' \in K_{l-1}(l \tau)} \!\! \int_{0}^{l \tau - (s_2 + \cdots + s_l)} \!  \! \norm{  \partial_{\nu}^{k+l}  e^{\nu \Delta_d}f\Big|_{\nu=t+s_1+\dots+s_l}}_p\hspace{-0.17cm}ds_1 ds' \\
        & = \int_{s' \in K_{l-1}(l \tau)} \int_{s_2 + \cdots + s_l}^{l \tau}
           \norm{  \partial_{\nu}^{k+l}  e^{\nu \Delta_d}f\Big|_{\nu=t+u}}_pdu ds' \\
        & = \int_0^{l \tau}  \norm{  \partial_{\nu}^{k+l}  e^{\nu \Delta_d}f\Big|_{\nu=t+u}}_p 
            \int_{s' \in K_{l-1}(u)} ds' \, du \\
        & \le C\int_0^{l \tau}  \norm{  \partial_{\nu}^{k+l}  e^{\nu \Delta_d}f\Big|_{\nu=t+u}}_p u^{l-1} \, du.
        \end{aligned}
    \]
  %  We give the bound of the $p$-norm using this last bound and applying Tonelli's theorem again,
  Therefore,
    \[
        \begin{aligned}
            \norm{\partial_t^k e^{t\Delta_d}(-\Delta_d)^\beta f}_p  & \leq  C
                                                                            \int_0^\infty
                                                                                \int_0^{l \tau}
                                                                                    \norm{  \partial_{\nu}^{k+l}  e^{\nu \Delta_d}f\Big|_{\nu=t+u}}_p 
                                                                                    u^{l-1} \, du
                                                                                \frac{d\tau}{\tau^{1+\beta}} \\
                                                                    & =    C
                                                                            \int_0^\infty
                                                                                \norm{  \partial_{\nu}^{k+l}  e^{\nu \Delta_d}f\Big|_{\nu=t+u}}_p 
                                                                                \int_\frac{u}{l}^{\infty}
                                                                                    \frac{d\tau}{\tau^{1+\beta}}
                                                                                u^{l-1} \, du \\
                                                                       &=C
                                                                            \int_0^\infty
                                                                                \norm{  \partial_{\nu}^{k+l}  e^{\nu \Delta_d}f\Big|_{\nu=t+u}}_p 
                                                                                \frac{u^{l-\beta}}{u}  du\\
                                                                                  & =C
                                                                            \int_t^\infty
                                                                                \norm{ \partial_{\nu}^{k+l}  e^{\nu \Delta_d}f}_p \frac{(\nu-t)^{l-\beta}}{\nu-t} \, d\nu.
        \end{aligned}
    \]
    When $1 \leq q < \infty$, we can use \Cref{Hardy_convolution_inequality} and \Cref{remark:amount_of_derivatives_on_t} to get %\Cref{lemma:inequality_remove_derivative} several times to obtain,
    \[
        \begin{aligned}
            \norm{t^{k-\frac{\alpha-2\beta}{2}} \partial_t^k e^{t \Delta_d} (-\Delta_d)^{\beta}f}_{p,q} 
                & \leq C \norm{t^{k+l-\frac{\alpha}{2}} \partial_t^{k+l} e^{t \Delta_d} f}_{p,q}< \infty. %, \\
             %   & \leq C \norm{t^{m-\frac{\alpha}{2}} \partial_t^m \blue{e^{t \Delta_d} f}}_{p,q} < \infty,
        \end{aligned}
    \]
 %   with $m = \left[ \frac{\alpha}{2}\right] +1$. 
 The case $q=\infty$ follows by using similar ideas to the ones in the proof of \Cref{theorem:BesselPotentials}.

\edproof

\newpage

\section{Appendix}\label{Appendix}

In this appendix we collect some known Hardy type inequalities that we use several times along the paper.

Next inequalities can be found in \cite[Chapter IX, Section 9.9, Theorem 329, Eq. (9.9.8) and (9.9.9)]{HardyLittlewoodPolya1952}, \cite{Muckenhoupt} and \cite[A.4, Appendix A]{Stein1970}.

\begin{lemma}
    \label{Hardy_inequality}
    Let $1\le p<\infty,$ $r>0,$ and $f$ be a non-negative measurable function $f.$ Then 
    \[
        \left(\int_0^\infty\left( \int_0^x f(y) \, dy\right)^p x^{-r-1} \, dx \right)^{\frac{1}{p}}
        \leq \frac{p}{r} \left( \int_0^\infty \left(y f(y) \right)^p y^{-r-1} \, dy \right)^{\frac{1}{p}},
    \]
    \[
    \left(\int_0^\infty\left( \int_x^\infty f(y) \, dy\right)^p x^{r-1} \, dx \right)^{\frac{1}{p}}
    \leq \frac{p}{r} \left( \int_0^\infty \left(y f(y) \right)^p y^{r-1} \, dy \right)^{\frac{1}{p}}.
    \]
\end{lemma}
\begin{comment}
\begin{proof}
    For the first inequality consider the kernel $K(x,y):= \chi_{[0,x]}(y)y^{\frac{r}{p}+\frac{1}{p}-1}x^{-\frac{r}{p}-\frac{1}{p}}$
    and the function $g(y):= f(y) y^{-\frac{r}{p}-\frac{1}{p}+1}$. Notice that the kernel is homogeneous and 
    the integral
    \[
        \int_0^\infty K(1,y)y^{-\frac{1}{p}}\, dy = \int_0^1 y^{\frac{r}{p}-1}\, dy = \frac{p}{r}.
    \]
   Therefore,
    \[
        \left(\int_0^\infty\left( \int_0^x f(y) \, dy\right)^p x^{-r-1} \, dx \right)^{\frac{1}{p}} 
        \leq \frac{p}{r} \left( \int_0^\infty \left(y f(y) \right)^p y^{-r-1} \, dy \right)^{\frac{1}{p}}.  
    \]
    For the second Hardy's inequality, we consider the kernel $K(x,y):= \chi_{[x,\infty]}(y)y^{-\frac{r}{p}+\frac{1}{p}-1}x^{\frac{r}{p}-\frac{1}{p}}$
    and the function $g(y):= f(y) y^{\frac{r}{p}-\frac{1}{p}+1}$ and we follow the same steps.
\end{proof}
\end{comment}

The following Hardy type convolution inequality appears in \cite[Chapter IX, Section 9.9, Theorem 329, Eq. (9.9.7)]{HardyLittlewoodPolya1952}.

\begin{lemma}\label{Hardy_convolution_inequality}
    Let $1\le p<\infty,$ $\alpha,\beta >0,$ and $f$ be a non-negative measurable function $f.$ Then
    \[
    \left(
        \int_0^\infty
            \left( \int_x^\infty f(y) \frac{(y-x)^\beta}{(y-x)} \, dy\right)^p x^{\alpha-1}
        \, dx
    \right)^{\frac{1}{p}}
    \leq 
    \frac{\Gamma(\beta) \Gamma\left( \frac{\alpha}{p}\right) }{\Gamma\left( \beta + \frac{\alpha}{p}\right)}
    \left(
        \int_0^\infty
            \left(y^\beta f(y) \right)^p y^{\alpha-1}
        \, dy
    \right)^{\frac{1}{p}}.
    \]
\end{lemma}

Finally we include the following discrete weighted Hardy inequalities. They can be found in \cite[Section2, Theorem 1, Eq. (2.14)]{HL}, see also \cite[Eq. (1') and (2'')]{Leindler}. %The second one is an easy consequence of the second inequality in Lemma \ref{Hardy_inequality} by using step functions.

\begin{lemma}
    \label{Hardy_discrete_inequality}
    Let $1\le p<\infty$ and $r>0$. There is $C>0$ such that for every  positive sequence $\{a_n\}_{n\in \NN}$, it holds that
    \[
        \left(\sum_{n=1}^\infty\left( \sum_{j=1}^n a_j \right)^p n^{-r-1} \right)^{\frac{1}{p}}
        \leq C \left( \sum_{n=1}^\infty \left(n a_n \right)^p n^{-r-1} \right)^{\frac{1}{p}},
    \]
    \[
        \left(\sum_{n=1}^\infty\left( \sum_{j=n+1}^\infty a_j \right)^p n^{r-1} \right)^{\frac{1}{p}}
        \leq C
         \left( \sum_{n=1}^\infty \left(n a_n \right)^p n^{r-1} \right)^{\frac{1}{p}}.
    \]
    %with $C_1 = \frac{p}{r} \max\{1,2^{\frac{p-r-1}{p}}\} 3^{\frac{r+1}{p}} $ and $C_2 = \frac{p}{r} \max\{1,2^{\frac{1-r}{p}}\} 2^{p+r-1}$.
\end{lemma}

\printbibliography %Prints bibliography

%\bibliographystyle{siam}
%\bibliography{references}

\end{document}